\pdfoutput=1
          \documentclass{cmslatex}
\usepackage[paperwidth=7in, paperheight=10in, margin=.875in]{geometry}
\usepackage{pdfpages}
\usepackage{amssymb}
\usepackage{empheq}
\usepackage{cases}
\usepackage{amsmath}
\usepackage{caption,lipsum}
\usepackage{stmaryrd}
\usepackage{graphicx,epsfig}
\usepackage{epstopdf}
\usepackage{tabularx}
\usepackage{color}
\usepackage{empheq,float}
\usepackage{tikz}
\usetikzlibrary{matrix}
\DeclareGraphicsExtensions{-eps-converted-to.pdf}
\input psfig.sty
\def\CC{{\mathbb C}}
\def\RR{{\mathbb R}}
\def\ZZ{{\mathbb Z}}
\def\TT{{\mathbb T}}
\def\calO{\mathcal O}

\newcommand{\fe}{\mathrm{e}}
\def\ds{\displaystyle}

\def\eps{\varepsilon}

\def\pa{\partial}

\newcommand{\be}{\begin{equation}}
\newcommand{\ee}{\end{equation}}
\newcommand{\ba}{\begin{array}}
\newcommand{\ea}{\end{array}}
\newcommand{\bea}{\begin{eqnarray}}
\newcommand{\eea}{\end{eqnarray}}
\newcommand{\beas}{\begin{eqnarray*}}
\newcommand{\eeas}{\end{eqnarray*}}

\numberwithin{equation}{section}
\renewcommand{\theequation}{\arabic{section}.\arabic{equation}}

\begin{document}
          \title{Uniformly accurate numerical schemes for the nonlinear Dirac equation in the nonrelativistic limit regime}


          \author{Mohammed Lemou\thanks{CNRS, Universit\'{e} de Rennes 1, IRMAR and INRIA-Rennes, Campus de Beaulieu, 35042 Rennes Cedex, France.({\tt mohammed.lemou@univ-rennes1.fr})}
         \and Florian M\'{e}hats \thanks{Universit\'{e} de Rennes 1, IRMAR and INRIA-Rennes, Campus de Beaulieu, 35042 Rennes Cedex, France. ({\tt florian.mehats@univ-rennes1.fr})}
         \and Xiaofei Zhao \thanks{Universit\'{e} de Rennes 1, IRMAR, Campus de Beaulieu, 35042 Rennes Cedex, France. ({\tt zhxfnus@gmail.com})}}





         \pagestyle{myheadings} \markboth{UA schemes for the nonlinear Dirac equation}{M. LEMOU, F. M\'{E}HATS AND X. ZHAO} \maketitle

          \begin{abstract}
          We apply the two-scale formulation approach to propose uniformly accurate (UA) schemes for solving the nonlinear Dirac equation in the nonrelativistic limit regime. The nonlinear Dirac equation involves two small scales $\eps$ and $\eps^2$ with $\eps\to0$ in the nonrelativistic limit regime. The small parameter causes high oscillations in time which brings severe numerical burden for classical numerical methods. We transform our original problem as a two-scale formulation and present a general strategy to tackle a class of highly oscillatory problems involving the two small scales $\eps$ and $\eps^2$. Suitable initial data for the two-scale formulation is derived to bound the time derivatives of the augmented solution. Numerical schemes with uniform (with respect to $\eps\in (0,1]$) spectral accuracy in space and uniform first order or second order accuracy in time are proposed. Numerical experiments are done to confirm the UA property.
          \end{abstract}

         \section{Introduction}
\label{sec:intro}
The nonlinear Dirac equation has been widely considered in many physical and mathematical areas, such as the electron self-interacting \cite{ele}, the gravity theory \cite{grav}, and recent studies in material graphene and Bose-Einstein condensates \cite{Ablowitz,Weistein,BEC}. In this paper, we consider the nonlinear reduced Dirac equation
\begin{equation}
\label{nlsw}
i\pa_t \Phi^\eps=-\frac{i}{\eps}\alpha \pa_x \Phi^\eps+\frac{1}{\eps^2}\beta \Phi^\eps+\left[V_e(t)+V_m(t)\alpha\right]\Phi^\eps+\lambda\left(\beta \Phi^\eps,\Phi^\eps\right)\beta \Phi^\eps,
\end{equation}
for $ t>0,\ x\in\RR$, where the unknown $\Phi^\eps=(\phi_1^\eps,\phi_2^\eps)^T=\Phi^\eps(t,x)$ is a bi-dimensional complex-valued column vector and is interpreted as the wave function subject to the initial condition
$$\Phi^\eps(0,x)=\Phi_0(x),\qquad x\in\RR.$$
$\lambda\in\RR$ denotes the coupling constant and $\eps\in(0,1]$ is a dimensionless parameter which is inversely proportional to the speed of light. 
$V_e(t)=V_e(t,x)$ and $V_m(t)=V_m(t,x)$ are two given real-valued scalar functions representing the electrical potential and magnetic potential \cite{Mauser,Jin,BL}, respectively, and the matrices $\alpha$ and $\beta$ are known as the Pauli matrices, i.e.
$$\alpha=\begin{pmatrix}0&1\\ 1&0\end{pmatrix},\qquad \beta=\begin{pmatrix}1&0\\ 0&-1\end{pmatrix}.$$
In particular, we have $(\beta \Phi^\eps,\Phi^\eps)=|\phi_1^\eps|^2-|\phi_2^\eps|^2$.
The nonlinear Dirac equation (\ref{nlsw}) has been widely considered in the literature \cite{Bao1,ESAIM,Tang3,Pecher} as a reduced mathematical model. It was originally derived in \cite{soler} as a four-component equations in 3D for describing the spinor field with nonlinear coupling. A more complicated cubic nonlinearity could be considered as in \cite{Thirring,Tang2}. We refer to \cite{Mauser,Bao1} for the nondimensionalization and the model reduction to the two-component problem (\ref{nlsw}).

The model (\ref{nlsw}) conserves the mass $M(t):=\int_{\RR}|\Phi^\eps|^2dx$ and the energy $E(t)$, provided that the two external potentials $V_e,V_m$ are independent of time,
\begin{align*}
 &\quad E(t)\\
 &:=\int_{\RR}\left(\frac{1}{\eps^2}(\Phi^\eps,\beta\Phi^\eps)-\frac{i}{\eps}(\Phi^\eps,\alpha\partial_x\Phi^\eps)+V_e|\Phi^\eps|^2+V_m(\Phi^\eps,\alpha
 \Phi^\eps)+\lambda(\beta\Phi^\eps,\Phi^\eps)|\Phi^\eps|^2\right)dx.
\end{align*}
With a fixed $\eps>0$, the well-posedness of the Cauchy problem (\ref{nlsw}) has been studied and established. We refer to \cite{ESAIM,Huh,Najman,Pecher,Bartsch} and the references therein for detailed analytical results. Various numerical methods including finite difference time domain (FDTD) methods and operator splitting methods have also been considered in \cite{Tang1,Tang2,Tang3,ESAIM,Added1} for solving the nonlinear Dirac equation in the classical regime, i.e. $\eps\approx 1$.

When $\eps\to0$, which corresponds to the speed of light going to infinity and is known as the nonrelativistic limit, the nonlinear Dirac equation (\ref{nlsw}) has been shown to converge in the energy space to the nonlinear Schr\"{o}dinger equations \cite{Mauser,Najman}. Both the analytical results and the numerical studies \cite{Mauser,Najman,Bao1,Bao2} show that the solution $\Phi^\eps$ of (\ref{nlsw}) propagates waves with wavelength $O(\eps^2)$ in time $t$ and wavelength $O(1)$ in space $x$, as $\eps\ll1$. The small temporal wavelength makes the solution highly oscillatory in time, and as a consequence it causes severe numerical burden for classical numerical discretizations. Results in \cite{Bao1,Bao0} show that, in order to capture the correct solution in the limit regime, one needs to use time step $\Delta t=O(\eps^2)$ and mesh size $\Delta x=O(\sqrt{\eps})$ for FDTD, and $\Delta t=O(\eps^2),\,\Delta x=O(1)$ for exponential integrators or time splitting spectral methods. Thus designing numerical methods that allow the use of step size independent of $\eps$ becomes a very needed and challenging issue. Recently, in study of the nonlinear Klein-Gordon equation in nonrelativistic limit regime, schemes with uniformly accurate (UA) properties have been proposed in \cite{BCZ,BDZ,cclm}. The UA schemes can use $\Delta t=O(1),\,\Delta x=O(1)$ for corrected approximations when $\eps\ll1$ and are significantly better than classical methods in the limit regime and asymptotic preserving (AP) methods in the intermediate regime, as shown in \cite{BZ,zhao}. Among the UA schemes, there are two approaches so far. One is based on the multiscale expansion of the solution \cite{BCZ,BDZ}, and the other uses the two-scale formulation approach \cite{cclm,clm}. The first approach strongly relies on the pre-knowledge of the expansion and the polynomial type nonlinearity. In fact, a multiscale approach  has already been considered for the linear Dirac in \cite{Bao2}, leading to a first order UA scheme. However, this proposed scheme and its corresponding numerical analysis become rather complicated in the nonlinear case. In contrast, the two-scale formulation approach provides a general strategy to design UA schemes for a class of oscillatory problems \cite{clm}. Nevertheless, the nonlinear Dirac equation (\ref{nlsw}) which contains two scales $\eps$ and $\eps^2$, does not belong to the type of problems in \cite{clm}, which means that the strategy in \cite{cclm,clm} cannot be directly applied. In fact, problems involving two scales are usual in collisional kinetic equations in the diffusion limit, but the presence of these two scales ($\eps$ and $\eps^2$) also appears in the context of oscillatory kinetic equations, as mentioned in \cite{clm,kinetic}. This extra scaling in the equations causes more difficulties for achieving UA property and  a suitable two-scale formulation has to be found.

In this work, we are going to propose a two-scale formulation and design UA schemes for solving the nonlinear Dirac equation in the nonrelativistic limit regime. A suitable formulation is presented on which we construct a second order in time UA scheme. Extensive numerical experiments are done to show the UA property. Extension of this strategy to the case of oscillatory kinetic equations with strong magnetic field is the subject of a forthcoming paper.

The rest of the paper is organized as follows. In Section \ref{sec:two-scale},
we introduce the two-scale formulation and construct the suitable initial data for this augmented problem. In Section \ref{sec:method},
we present the numerical schemes based on this two-scale formulation.  Numerical results are reported in
Section \ref{sec:result} and some conclusions are finally drawn in Section \ref{sec:conc}.

\section{Two scale formulation}\label{sec:two-scale}
To filter out the main oscillation in \eqref{nlsw}, let us introduce the filtered data
\begin{equation}\label{filter}
u^\eps=\begin{pmatrix}\fe^{it/\eps^2}&0\\ 0&\fe^{-it/\eps^2}\end{pmatrix}\Phi^\eps.
\end{equation}
It solves the equation
\begin{equation}\label{filtered}
\pa_t u^\eps=-\frac{1}{\eps} A(t/\eps^2)\pa_x u^\eps+F(t,t/\eps^2,u^\eps),\quad t>0,\ x\in\RR,
\end{equation}
with
$$A(\tau)= \begin{pmatrix}0&\fe^{2i\tau}\\ \fe^{-2i\tau}&0\end{pmatrix},$$
and
\begin{equation}\label{F def}
F(t,\tau,u^\eps)=-i\left[V_e(t)+V_m(t)A(\tau)\right]u^\eps-i\lambda\left(\beta u^\eps,u^\eps\right)\beta u^\eps.
\end{equation}
In this paper, we assume that $V_e$ and $V_m$ are $C^1$ functions of time.
For future convenience, we denote the nonlinearity without the magnetic potential by
$$F_e(t,u^\eps):=-i\left[V_e(t)u^\eps+\lambda\left(\beta u^\eps,u^\eps\right)\beta u^\eps\right].$$
Now, following \cite{cclm,clm}, we separate the fast time variable $\tau=t/\eps^2$ from the slow time variable $t$ in the solution $u^\eps(t,x)$ of (\ref{filtered}) and consider the two-scale formulation for the augmented unknown $U^\eps=U^\eps(t,\tau,x)$$\in\CC^2$:
\begin{equation}
\label{TS}
\pa_t U^\eps+\frac{1}{\eps^2}\pa_\tau U^\eps=-\frac{1}{\eps}A(\tau)\pa_x U^\eps+F(t,\tau,U^\eps),\quad t>0,\ \tau\in\TT,\ x\in\RR.
\end{equation}
Here $U^\eps(t,\tau,x)$ is $2\pi$-periodic in $\tau$ and $\TT:=\RR/(2\pi\ZZ)$ denotes the torus. With initial data $U^\eps(0,\tau,x)$ satisfying
$$U^\eps(0,0,x)=u^\eps(0,x)=\Phi_0(x),\quad x\in\RR,$$
the solution of the two-scale problem (\ref{TS}) enables to recover the solution of the filtered equation (\ref{filtered}) by setting
$$U^\eps\left(t,\frac{t}{\eps^2},x\right)=u^\eps(t,x),\quad t\geq0,\ x\in\RR.$$

In the two-scale formulation (\ref{TS}), $\tau$ is considered as an additional independent variable, which is periodic. The advantage one can get from this formulation is that now the initial data $U^\eps(0,\tau,x)$ is only prescribed at a single point $\tau=0$, so there is some freedom to choose the initial data. By choosing a suitable initial data, we will succeed to bound the time derivatives of $U^\eps$ uniformly with respect to $\eps$, which allows to design uniformly accurate numerical schemes. 

\subsection{A toy model}
In order to find out a suitable initial data for problems like \eqref{TS}, we propose to analyze the following simple toy model:
\begin{equation}
\label{toy}
\partial_tu+\frac{1}{\eps^2}\partial_\tau u+i\frac{a(\tau)}{\eps}u=0,\qquad \mbox{with}\quad u(0,\tau)=u_{in}(\tau).
\end{equation}
Here, $u(t,\tau)$ is a scalar unknown, $t\geq 0$ and $\tau\in \TT$. Moreover, we assume (as for the augmented problem \eqref{TS}) that only $u_{in}(0)=u_0$ is prescribed, that $a$ is a real-valued smooth periodic function and that we have the property
$$\int_0^{2\pi}a(\tau)d\tau=0.$$ Our aim is to derive expressions for $u_{in}(\tau)$ in order to ensure conditions (i) $u_{in}(0)=u_0$, and (ii) that time-derivative of $u$, up to some order $p\geq 1$, are uniformly bounded with respect to $\eps$, a property which enables to construct easily uniformly accurate numerical schemes for \eqref{toy}. 

The exact solution of \eqref{toy} is given by
$$u(t,\tau)=\fe^{-i\eps b(\tau)}\fe^{i\eps b(\tau-t/\eps^2)}u_{in}(\tau-t/\eps^2),$$
where $b(\tau)=\int_0^\tau a(s)ds$. Note that $b$ is periodic, since the average of $a$ vanishes. It is clear that
$$\frac{\partial^k u}{dt^k}(t,\tau)=\fe^{-i\eps b(\tau)}\frac{(-1)^k}{\eps^{2k}}\left.\frac{d^k}{ds^k}\left(\vphantom{\int}\fe^{i\eps b(s)}u_{in}(s)\right)\right|_{s=\tau-t/\eps^2}.$$
Therefore, in order to get uniformly bounded derivatives for $k=0,\ldots,p$, we must have
$$\frac{d^p}{ds^p}\left(\vphantom{\int}\fe^{i\eps b(s)}u_{in}(s)\right)=\mathcal O(\eps^{2p}),$$
which is satisfied as soon as
$$u_{in}(\tau)=Q_{p-1}(\tau)\fe^{-i\eps b(\tau)}+\mathcal O(\eps^{2p}).$$
Here $Q_{p-1}$ is any polynomial of degree $\leq p-1$. Due to the smoothness and periodicity of $u_{in}$, we have necessarily that $Q$ is a constant polynomial, thus
$$u_{in}(\tau)=C\fe^{-i\eps b(\tau)}+\mathcal O(\eps^{2p}).$$
Therefore, the following initial data is suitable to ensure (i) and (ii):
$$u_{in}(\tau)=u_0+u_0\sum_{k=1}^{2p-1}\frac{(-ib(\tau))^k}{k!}\,\eps^k.$$
In particular, one observes that one needs to choose the initial data as an expansion in powers $\eps$ up to the order $\eps^{2p-1}$ in order to ensure the boundedness of the time derivatives of $u$ up to the order $p$. This crucial property will guide our analysis in the next subsection. Of course, for \eqref{TS}, we do not have an exact solution, but we will obtain iteratively the expansion in powers of $\eps$ of the suitable initial data by using Chapman-Enskog techniques.

\subsection{Suitable initial data for the augmented problem (\ref{TS})}


In order to define the initial data $U^\eps(0,\tau)$ where here and after we omit the space variable $x$ for simplicity, let us perform {\em formally} the Chapman-Enskog expansion of $U^\eps$ \cite{cclm,clm,zhao}. To this aim, we introduce the operators $L$ and $\Pi$ for a periodic function $h(\tau):\TT\to\CC^2$ as
$$Lh:=\partial_\tau h,\quad \Pi h:=\frac{1}{2\pi}\int_0^{2\pi}h(\tau)d\tau,$$
and when $\Pi h=0$, the operator $L$ is invertible with
$$(L^{-1}h)(\tau)=(I-\Pi)\int_0^\tau h(\theta)d\theta.$$
We perform the Chapman-Enskog expansion by setting
$$U^\eps(t,\tau)=\underline U(t)+h(t,\tau),\quad \mbox{with }\underline U=\Pi U^\eps.$$
We observe that $\Pi A=0$ and write the micro-macro formulation of \eqref{TS} as
\begin{align}
\pa_t \underline U&=-\frac{1}{\eps}\Pi\left(A(\tau)\pa_x h\right)+\Pi\left(F(t,\tau,\underline U+h)\right),\label{U def}\\
\pa_t h &=-\frac{1}{\eps^2}L h-\frac{1}{\eps}A(\tau)\pa_x \underline U-\frac{1}{\eps}(I-\Pi)\left(A(\tau)\pa_x h\right)+(I-\Pi)F(t,\tau,\underline U+h).\label{h def}
\end{align}

Based on this formulation, it appears that, when $\eps\to 0$, $U^\eps(t,\tau)$ converges to $U(t)$ satisfying
\begin{equation}\label{lm1}\pa_t U=C\partial_x^2U+\Pi \left(F(t,\tau,U)\right)=F_e(t,U),
\end{equation}
where the matrix $C$ is given below by \eqref{defC}.
In particular, when $\eps$ is small, the solution $\Phi^\eps$ of the nonlinear Dirac equation (\ref{nlsw}) will formally be close to $\Phi$ defined by
$$\Phi=\begin{pmatrix}\fe^{-it/\eps^2}&0\\ 0&\fe^{it/\eps^2}\end{pmatrix}U.$$
The above limit model (\ref{lm1}) is a system of coupled nonlinear Schr\"{o}dinger equations.
When there is no magnetic potential, i.e. $V_m=0$, the convergence of the model (\ref{filtered}) to (\ref{lm1}) as $\eps\to0$ has been proved rigorously in some energy space in \cite{Najman}. 

\subsubsection{Uniform boundedness of time derivatives up to order 2}
In this subsection, we derive the expression of $U^\eps(0,\tau)$ such that the time derivatives of $U^\eps$ up to order 2 are uniformly bounded. Following the previous subsection on the toy model, we need an expansion up to the order $\eps^3$. In the following formal calculations, we thus assume that derivatives of $h$ until order 2 are bounded.

By applying the inverse of $L$ to \eqref{h def}, we get
\begin{align}
h=&-\eps L^{-1}A\partial_x\underline{U}-\eps L^{-1}(I-\Pi)A\partial_xh+\eps^2L^{-1}(I-\Pi)F(t,\tau,\underline{U}+h)-\eps^2L^{-1}\partial_th.\label{h all}
\end{align}
Then we further have
\begin{align}
\partial_xh=&-\eps L^{-1}A\partial_x^2\underline{U}-\eps L^{-1}(I-\Pi)A\partial_x^2h+\eps^2L^{-1}(I-\Pi)\partial_xF(t,\tau,\underline{U}+h)\nonumber\\
&-\eps^2L^{-1}\partial_{tx}h,\label{hx all}\\
\partial_th=&-\eps L^{-1}A\partial_{tx}\underline{U}-\eps L^{-1}(I-\Pi)A\partial_{tx}h+\eps^2L^{-1}(I-\Pi)\big[\partial_tF(t,\tau,\underline{U}+h)\nonumber\\
&+\partial_uF(t,\tau,\underline{U}+h)(\partial_t\underline{U}+\partial_th)\big]-\eps^2L^{-1}\partial_{t}^2h.\label{ht all}
\end{align}
Hence, we can expand the unknowns in powers of $\eps$ and find
\begin{align*}
h=&-\eps L^{-1}A\pa_x\underline U+\eps^2 L^{-1}(I-\Pi)AL^{-1}A\pa^2_x\underline U+\eps^2L^{-1}(I-\Pi)F(t,\tau,\underline{U})+\calO(\eps^3).
\end{align*}
Therefore,
\begin{align*}
U^\eps(0,\tau)=&\underline U-\eps L^{-1}A\pa_x\underline U+\eps^2 L^{-1}(I-\Pi)AL^{-1}A\pa^2_x\underline U\\
&+\eps^2L^{-1}(I-\Pi)F(t,\tau,\underline{U})+\calO(\eps^3).
\end{align*}
Direct computations yield
$$B(\tau):=L^{-1}A(\tau)=-\frac{i}{2}\begin{pmatrix}0&\fe^{2i\tau}\\ -\fe^{-2i\tau}&0\end{pmatrix}$$
and
\begin{equation}\label{defC}C\equiv C(\tau):=A(\tau)B(\tau)=\frac{i}{2}\begin{pmatrix}1&0\\ 0&-1\end{pmatrix}
\end{equation}
so $(I-\Pi)C(\tau)=0$ and
\begin{align*}
U^\eps(0,\tau)=\underline U(0)&-\eps B(\tau)\pa_x\underline U(0)+\eps^2L^{-1}(I-\Pi)F(0,\tau,\underline{U})+\calO(\eps^3).
\end{align*}
Now we use that $U^\eps(0,0)=\Phi_0$ to get
\begin{align}\label{U 2}
\underline U(0)&=\Phi_0+\eps B(0)\pa_x\underline U(0)-\eps^2f_0(0)+\calO(\eps^3)\nonumber\\
&=\Phi_0+\eps B(0)\pa_x\Phi_0+\eps^2 B(0)^2\pa^2_x\Phi_0-\eps^2f_0(0)+\calO(\eps^3),
\end{align}
where
$$f_0(\tau):=L^{-1}(I-\Pi)F(0,\tau,\Phi_0)=-iV_m(0)B(\tau)\Phi_0,$$
and then
\begin{align*}
U^\eps(0,\tau)
=&\Phi_0-\eps \left(B(\tau)-B(0)\right)\pa_x\Phi_0-\eps^2 \left(B(\tau)-B(0)\right)B(0)\pa^2_x\Phi_0\\
&+\eps^2(f_0(\tau)-f_0(0))+\calO(\eps^3).
\end{align*}
We thus get an expression $U^\eps_2$ for the suitable initial data  up to $O(\eps^3)$, i.e. $U^\eps(0,\tau)=U^\eps_2(\tau)+O(\eps^3)$, as follows:
\begin{align*}U_{2}^\eps(\tau):=&\Phi_0+ \frac{i\eps }{2}\begin{pmatrix}0&\fe^{2i\tau}-1\\ 1-\fe^{-2i\tau}&0\end{pmatrix}\pa_x\Phi_0+\frac{\eps^2}{4}\begin{pmatrix}1-\fe^{2i\tau}&0\\ 0&1-\fe^{-2i\tau}\end{pmatrix}\pa^2_x\Phi_0\\
&-\frac{\eps^2}{2}V_m(0)\begin{pmatrix}0&\fe^{2i\tau}-1\\ 1-\fe^{-2i\tau}&0\end{pmatrix}\Phi_0.
\end{align*}
For future use, we denote $U_2(\tau)=U_{2}^\eps(\tau)-\eps^2(f_0(\tau)-f_0(0))$, and also define
$$U_1^\eps(\tau):=\Phi_0+ \frac{i\eps }{2}\begin{pmatrix}0&\fe^{2i\tau}-1\\ 1-\fe^{-2i\tau}&0\end{pmatrix},$$
which gives $U^\eps(0,\tau)=U^\eps_1(\tau)+O(\eps^2)$.

Now we push the asymptotic expansion to the next order.
Inserting (\ref{hx all}) and (\ref{ht all}) into (\ref{h all}) and letting $t=0$, we get
\begin{align}
h(0,\tau)&=-\eps B\partial_x\underline{U}-\eps L^{-1}(I-\Pi)A\left[-\eps B\partial_x^2\underline{U}+\eps^2L^{-1}(I-\Pi)C\partial_x^3\underline{U}\right.\nonumber\\
&\quad\left.+\eps^2L^{-1}(I-\Pi)\partial_xF(0,\tau,\underline{U})\right]
+\eps^2L^{-1}(I-\Pi)F(0,\tau,\underline{U}+h)\nonumber\\
&\quad+\eps^3L^{-1}B\partial_{tx}\underline{U}+O(\eps^4)\nonumber\\
&=-\eps B\partial_x\underline{U}+\eps^2L^{-1}(I-\Pi)F(0,\tau,U_1^\eps)+\eps^3L^{-1}B\partial_{tx}\underline{U}\nonumber\\
&\quad-\eps^3 L^{-1}(I-\Pi)AL^{-1}(I-\Pi)\partial_xF(0,\tau,\underline{U})+O(\eps^4).
\label{h 3}
\end{align}
From (\ref{U def}) and (\ref{hx all}), we have
\begin{equation}\label{limit model}
\partial_t\underline{U}= C\partial_x^2\underline{U}+\Pi F(t,\tau,\underline{U})+O(\eps).
\end{equation}
Noting $F_e(t,\underline{U})=\Pi F(t,\tau,\underline{U})$, we find
\begin{align*}
h(0,\tau)=&-\eps B\partial_x\underline{U}+\eps^2L^{-1}(I-\Pi)F(0,\tau,U_1^\eps)+\eps^3L^{-1}B
\left[C\partial_{x}^3\underline{U}+\partial_xF_e(0,\underline{U})\right]\\
&-\eps^3 L^{-1}(I-\Pi)AL^{-1}(I-\Pi)\partial_xF(0,\tau,\underline{U})+O(\eps^4).
\end{align*}
Using (\ref{U 2}), we then get
\begin{align}
h(0,\tau)=&-\eps B(\tau)\left[\partial_x\Phi_0+\eps B(0)\partial_x^2\Phi_0+\eps^2B^2(0)\partial_x^3\Phi_0-\eps^2\partial_xf_0(0)\right]
+\eps^2f_1(\tau)\nonumber\\
&+\eps^3L^{-1}B\left[C\partial_{x}^3\Phi_0+\partial_xF_e(0,\Phi_0)\right]-\eps^3L^{-1}(I-\Pi)A\partial_xf_0+O(\eps^4),\label{h 30}
\end{align}
where
$$f_1(\tau):=L^{-1}(I-\Pi)F(0,\tau,U_1^\eps).$$
Based on
$$U^\eps(0,\tau)=\underline{U}(0)+h(0,\tau)=\Phi_0+h(0,\tau)-h(0,0),$$
and remarking that $L^{-1}B=-A/4$ and $L^{-1}(I-\Pi)A\partial_xf_0\equiv0$, we can update the expansion of $\underline{U}(0)$ as
\begin{align}\label{U 3}
\underline{U}(0)=&\Phi_0+\eps B(0)\partial_x\Phi_0+\eps^2B^2(0)\partial_x^2\Phi_0-\eps^2f_1(0)+\eps^3B^3(0)\partial_x^3\Phi_0\nonumber\\
&-\eps^3B(0)\partial_xf_0(0)
+\frac{\eps^3}{4}A(0)\left[C\partial_x^3\Phi_0+\partial_xF_e(0,\Phi_0)\right]+O(\eps^4),
\end{align}
and then get an expression $U^\eps_3$ of the prepared initial data up to $O(\eps^4)$, i.e. $U^\eps(0,\tau)=U^\eps_3(\tau)+O(\eps^4)$, as follows:
\begin{align}
U^\eps_3(\tau)=&U_2(\tau)+\eps^2\left[f_1(\tau)-f_1(0)\right]-\eps^3\left[B(\tau)-B(0)\right]B^2(0)\partial_x^3\Phi_0\\
&-\frac{\eps^3}{4}(A(\tau)-A(0))
\left[C\partial_{x}^3\Phi_0+\partial_xF_e(0,\Phi_0)\right]+\eps^3(B(\tau)-B(0))\partial_xf_0(0)\nonumber\\
=&U_2(\tau)+\eps^2\left[f_1(\tau)-f_1(0)\right]+\frac{i\eps^3}{4}\begin{pmatrix}0&\fe^{2i\tau}-1\\ 1-\fe^{-2i\tau}&0\end{pmatrix}\partial_x^3\Phi_0\nonumber\\
&+\frac{i\eps^3}{4}\begin{pmatrix}1-\fe^{2i\tau}&0\\ 0&1-\fe^{-2i\tau}\end{pmatrix}\partial_x(V_m(0)\Phi_0)\nonumber\\
& +\frac{\eps^3}{4}\begin{pmatrix}0&1-\fe^{2i\tau}\\ 1-\fe^{-2i\tau}&0\end{pmatrix}\partial_xF_e(0,\Phi_0).\label{U3eps}
\end{align}


\subsubsection{Uniform boundedness of time derivatives up to order 3}
Repeating the procedure above, we are going to get the next order expansion. Here we assume that derivatives of $h$ until order 3 are bounded. From (\ref{ht all}), we have
$$\partial_th=-\eps B\partial_{tx}\underline{U}+\eps^2L^{-1}(I-\Pi)\left[\partial_tF(t,\tau,\underline{U})
+\partial_uF(t,\tau,\underline{U})\partial_t\underline{U}\right]+O(\eps^3).$$
Plugging the above expansion and (\ref{h 3}) into (\ref{h all}) and letting $t=0$, we get
\begin{align}
h(0,\tau)&=-\eps B\partial_x\underline{U}+\eps L^{-1}(I-\Pi)A\left[\eps B\partial_x^2\underline{U}-\eps^2f_1(\tau)-\eps^3L^{-1}B\partial_{txx}\underline{U}
\right]\nonumber\\
&\quad+\eps^2L^{-1}(I-\Pi)F(0,\tau,U^\eps_2)
+\eps^3L^{-1}B\partial_{tx}\underline{U}-\eps^4f_t(\tau)+O(\eps^5)\nonumber\\
&=-\eps B\partial_x\underline{U}+\eps^2L^{-1}(I-\Pi)F(0,\tau,U^\eps_2)-\eps^3 L^{-1}(I-\Pi)Af_1\nonumber\\
&\quad-\frac{\eps^3}{4}A\partial_{tx}\underline{U}-\eps^4f_t(\tau)+O(\eps^5),
\label{h 4}
\end{align}
where
\begin{align*}
f_t(\tau)&:=L^{-2}(I-\Pi)\left[\partial_tF(0,\tau,\Phi_0)
+\partial_uF(0,\tau,\Phi_0)(C\partial_x^2\Phi_0+F_e(0,\Phi_0))\right]\\
&=\frac{i}{4}A(\tau)\left[\partial_tV_m(0)\Phi_0+V_m(0)(C\partial_x^2\Phi_0+ F_e(0,\Phi_0))\right].
\end{align*}
From (\ref{U def}), we find
\begin{align*}
\partial_t\underline{U}(0)&=C\partial_x^2\underline{U}+\Pi F(0,\tau,U_1^\eps)-\eps\Pi A\partial_xf_1+O(\eps^2).
\end{align*}
Then (\ref{h 4}) becomes
\begin{align}
h(0,\tau)=&-\eps B(\tau)\partial_x\underline{U}+\eps^2L^{-1}(I-\Pi)F(0,\tau,U^\eps_2)-\eps^3 L^{-1}(I-\Pi)Af_1\nonumber\\
&-\frac{\eps^3}{4}A(\tau)\left[C\partial_x^3\underline{U}+\Pi\partial_xF(0,\tau,U_1^\eps)-\eps\Pi A\partial_x^2f_1\right]-\eps^4f_t(\tau)+O(\eps^5).\label{h 41}
\end{align}
Combining (\ref{h 41}) with (\ref{U 3}) and noting $A(0)C=4B^3(0)$, we get
\begin{align}
h(0,\tau)
=&-\eps B(\tau)\big[\partial_x\Phi_0+\eps B(0)\partial_x^2\Phi_0+\eps^2B^2(0)\partial_x^3\Phi_0-\eps^2\partial_xf_1(0)+2\eps^3B^3(0)\partial_x^4\Phi_0\nonumber\\
&\quad-\eps^3B(0)\partial_x^2f_0(0)+\frac{\eps^3}{4}A(0)\partial_x^2F_e(0,\Phi_0)\big]+\eps^2L^{-1}(I-\Pi)F(0,\tau,U^\eps_2)\nonumber\\
&\quad-\eps^3 L^{-1}(I-\Pi)Af_1
-\frac{\eps^3}{4}A(\tau)C\left[\partial_x^3\Phi_0+\eps B(0)\partial_x^4\Phi_0\right]\nonumber\\
&\quad-\frac{\eps^3}{4}A(\tau)\left[\Pi\partial_xF(0,\tau,U_1^\eps)-\eps\Pi A\partial_x^2f_1\right]-\eps^4f_t(\tau)+O(\eps^5).\label{h 40}
\end{align}
Then we have
\begin{align}\label{U 4}
\underline{U}(0)=&\Phi_0+\eps B(0)\partial_x\Phi_0+\eps^2B^2(0)\partial_x^2\Phi_0-\eps^2f_2^\eps(0)+2\eps^3B^3(0)\partial_x^3\Phi_0-\eps^3B(0)\partial_xf_1(0)\nonumber\\
&+\frac{\eps^3}{4}A(0)\Pi\partial_xF(0,\tau,U_1^\eps)+\eps^3g^\eps(0)+3\eps^4B^4(0)\partial_x^4\Phi_0-\eps^4B^2(0)\partial_x^2f_0(0)\nonumber\\
&+\frac{\eps^4}{4}B(0)A(0)\partial_x^2F_e(0,\Phi_0)-\frac{\eps^4}{4}A(0)\Pi A\partial_x^2f_1+\eps^4f_t(0)+O(\eps^5),
\end{align}
where
\begin{align*}
f_2(\tau):=L^{-1}(I-\Pi)F(0,\tau,U^\eps_2),\quad g_1(\tau):=L^{-1}(I-\Pi)Af_1.
\end{align*}
Thus, we get the prepared initial data $U^\eps(0,\tau)=U_4^\eps(\tau)+O(\eps^5)$, with
\begin{align}
 U^\eps_4(\tau)
 &=U_2(\tau)+\eps^2\left[f_2(\tau)-f_2(0)\right]+\frac{i\eps^3}{4}\begin{pmatrix}0&\fe^{2i\tau}-1\\ 1-\fe^{-2i\tau}&0\end{pmatrix}\partial_x^3\Phi_0\label{Ueps 4}\\
 &+\frac{\eps^3}{4}\begin{pmatrix}0&1-\fe^{2i\tau}\\ 1-\fe^{-2i\tau}&0\end{pmatrix}\Pi\partial_xF(0,\tau,U_1^\eps)-\eps^3\left[g_1(\tau)-g_1(0)\right]\nonumber\\
 &-\frac{i\eps^3}{2}\begin{pmatrix}0&\fe^{2i\tau}-1\\ 1-\fe^{-2i\tau}&0\end{pmatrix}\partial_xf_1(0)-\frac{3\eps^4}{16}\begin{pmatrix}\fe^{2i\tau}-1&0\\ 0&\fe^{-2i\tau}-1\end{pmatrix}\partial_x^4\Phi_0\nonumber\\
 &-\frac{\eps^4}{8}\begin{pmatrix}0&\fe^{2i\tau}-1\\ 1-\fe^{-2i\tau}&0\end{pmatrix}\partial_x^2(V_m(0)\Phi_0)
 +\frac{\eps^4}{4}\begin{pmatrix}0&\fe^{2i\tau}-1\\ \fe^{-2i\tau}-1&0\end{pmatrix}\Pi A\partial_x^2f_1\nonumber\\
 &-\frac{i\eps^4}{8}\begin{pmatrix}1-\fe^{2i\tau}&0\\ 0&\fe^{-2i\tau}-1\end{pmatrix}\partial_x^2F_e(0,\Phi_0)
 -\eps^4\left(f_t(\tau)-f_t(0)\right).\nonumber
\end{align}

To get the next order expansion, we recall that we assume $\partial_th,\partial_t^2h,\partial_t^3h=\mathcal O(1)$ as $\eps\to0$, which indicates that
\begin{align*}
\partial_th=&-\eps B\partial_{tx}\underline{U}+\eps^2L^{-1}(I-\Pi)\frac{d}{dt}F\left(t,\tau,\underline{U}(t)+h(t)\right)+\eps^3L^{-1}B\partial_{ttx}\underline{U}\\
&-\eps^3L^{-1}(I-\Pi)AL^{-1}(I-\Pi)\frac{d}{dt}\partial_xF\left(t,\tau,\underline{U}(t)\right)+O(\eps^4).
\end{align*}
Using
\begin{align*}
&\frac{d}{dt}F\left(t,\tau,\underline{U}(t)+h(t)\right)=\partial_tF(t,\tau,\underline{U}+h)+\partial_uF(t,\tau,\underline{U}+h)
(\partial_t\underline{U}+\partial_th)\\
&=\partial_tF(t,\tau,\underline{U}-\eps B\partial_x\underline{U})+\partial_uF(t,\tau,\underline{U}-\eps B\partial_x\underline{U})
(\partial_t\underline{U}-\eps B\partial_{tx}\underline{U})+O(\eps^2)
\end{align*}
and denoting
\begin{align*}
H_0(\tau):=&L^{-1}(I-\Pi)[\partial_tF(0,\tau,\Phi_0)+\partial_uF(0,\tau,\Phi_0)
C\partial_x^2\Phi_0+\partial_uF(0,\tau,\Phi_0)F_e(0,\Phi_0)],\\
H_1(\tau):=&\partial_tF(0,\tau,U_1^\eps)+\partial_uF(0,\tau,U_1^\eps)
C(\partial_x^2\Phi_0+\eps B(0)\partial_x^3\Phi_0)\\
&+\partial_uF(0,\tau,U_1^\eps)\left[\Pi F(0,\tau,U_1^\eps)-\eps\Pi A\partial_xf_1\right]\\
&-\eps \partial_uF(0,\tau,U_1^\eps)B(\tau)\left[C\partial_x^3\Phi_0+\partial_xF_e(0,\Phi_0)\right],
\end{align*}
we get
\begin{align}
\partial_t h(0,\tau)=&-\eps B(\tau)\partial_{tx}\underline{U}+\eps^2L^{-1}(I-\Pi)H_1+\eps^3L^{-1}B\partial_{ttx}\underline{U}\nonumber\\
&-\eps^3L^{-1}(I-\Pi)A\partial_xH_0+O(\eps^4).\label{ht 4}
\end{align}
A detailed computation gives
$$H_0(\tau)=-iB(\tau)\left[\partial_tV_m(0)\Phi_0+V_m(0)(C\partial_{x}^2\Phi_0+F_e(0,\Phi_0))\right].$$
Plugging (\ref{h 30}) into (\ref{U def}) at $t=0$, we have
\begin{align*}
\partial_t\underline{U}(0)=&C\partial_x^2\Phi_0+\Pi F(0,\tau,U_2^\eps)+\eps CB(0)\partial_x^3\Phi_0
+\frac{\eps^2}{2}C\partial_x^4\Phi_0-\eps\Pi A\partial_xf_1\\
&-\eps^2C\partial_x^2f_0(0)
+\frac{\eps^2}{4}\partial_x^2F_e(0,\Phi_0)+O(\eps^3),
\end{align*}
and
\begin{align*}
\partial_t^2\underline{U}(0)=&C^2\partial_x^4\Phi_0+C\partial_x^2F_e(0,\Phi_0)
+Z_e+O(\eps),
\end{align*}
where
$$Z_e=\partial_uF_e(0,\Phi_0)\left(C\partial_x^2\Phi_0+F_e(0,\Phi_0)\right)
+\partial_tF_e(0,\Phi_0).$$
Combining the above two identities with (\ref{ht 4}), and inserting them together with  (\ref{U 4}) and (\ref{h 40}) into (\ref{h all}), we finally obtain the  prepared initial data  $U^\eps(0,\tau)=U_5^\eps(\tau)+O(\eps^6)$, with
%
\begin{align}
 & U^\eps_5(\tau)=
 U_2(\tau)+\eps^2\left(f_3(\tau)-f_3(0)\right)+\frac{i\eps^3}{4}\begin{pmatrix}0&\fe^{2i\tau}-1\\ 1-\fe^{-2i\tau}&0\end{pmatrix}\partial_x^3\Phi_0\label{U 5}\\
 &+\frac{\eps^3}{4}\begin{pmatrix}0&1-\fe^{2i\tau}\\ 1-\fe^{-2i\tau}&0\end{pmatrix}\Pi\partial_xF(0,\tau,U_2^\eps)-\eps^3\left(g_2(\tau)-g_2(0)\right)\nonumber\\
 &-\frac{i\eps^3}{2}\begin{pmatrix}0&\fe^{2i\tau}-1\\ 1-\fe^{-2i\tau}&0\end{pmatrix}\partial_xf_2(0)+\frac{\eps^4}{4}\begin{pmatrix}\fe^{2i\tau}-1&0\\ 0&\fe^{-2i\tau}-1\end{pmatrix}\partial_x^2f_1(0)\nonumber\\
 &-\frac{i\eps^4}{8}\begin{pmatrix}1-\fe^{2i\tau}&0\\ 0&\fe^{-2i\tau}-1\end{pmatrix}\Pi\partial_x^2F(0,\tau,U_1^\eps)
 -\frac{3\eps^4}{16}\begin{pmatrix}\fe^{2i\tau}-1&0\\ 0&\fe^{-2i\tau}-1\end{pmatrix}\partial_x^4\Phi_0\nonumber\\
 &+\frac{i\eps^4}{2}\begin{pmatrix}0&\fe^{2i\tau}-1\\ 1-\fe^{-2i\tau}&0\end{pmatrix}\partial_xg_1(0)+\eps^4(v(\tau)-v(0))-\eps^4\left(w_1(\tau)-w_1(0)\right)\nonumber\\
 &+\frac{\eps^4}{4}\begin{pmatrix}0&\fe^{2i\tau}-1\\ \fe^{-2i\tau}-1&0\end{pmatrix}\Pi A\partial_x^2f_1+\frac{3i\eps^5}{16}\begin{pmatrix}0&\fe^{2i\tau}-1\\ 1-\fe^{-2i\tau}&0\end{pmatrix}\partial_x^5\Phi_0\nonumber\\
 &+\eps^5\left(w_0(\tau)-w_0(0)\right)-\frac{3\eps^5}{16}\begin{pmatrix}0&\fe^{2i\tau}-1\\ \fe^{-2i\tau}-1&0\end{pmatrix}\partial_x^3F_e(0,\Phi_0)\nonumber\\
 &+\frac{i\eps^5}{8}\begin{pmatrix}1-\fe^{2i\tau}&0\\ 0&\fe^{-2i\tau}-1\end{pmatrix}\Pi A\partial_x^3f_1-\frac{\eps^5}{8}\begin{pmatrix}\fe^{2i\tau}-1&0\\ 0&1-\fe^{-2i\tau}\end{pmatrix}\partial_xZ_m\nonumber\\
 &-\frac{i\eps^5}{8}\begin{pmatrix}\fe^{2i\tau}-1&0\\ 0&\fe^{-2i\tau}-1\end{pmatrix}\partial_x^3(V_m(0)\Phi_0)-\frac{i\eps^5}{8}\begin{pmatrix}0&\fe^{2i\tau}-1\\ 1-\fe^{-2i\tau}&0\end{pmatrix}\partial_x Z_e\nonumber
\end{align}
where
\begin{align*}
&f_3(\tau):=L^{-1}(I-\Pi)F(0,U^\eps_3),\qquad g_2(\tau)=L^{-1}(I-\Pi)A\partial_xf_2,\\
&v(\tau):=L^{-1}(I-\Pi)AL^{-1}(I-\Pi)A\partial_xf_1,\\
&Z_m=\partial_tV_m(0)\Phi_0+V_m(0)(C\partial_x^2\Phi_0+ F_e(0,\Phi_0)),\\
&w_0(\tau)=L^{-2}(I-\Pi)A\partial_xH_0,\qquad w_1(\tau):=L^{-2}(I-\Pi)H_1.\\
\end{align*}
The augmented problem \eqref{TS}, endowed with one of the initial data $U_{2k-1}^\eps$ ($k\in\{1,2,3\}$), provides a solution which has uniformly bounded time derivatives up to the order $k$, respectively.
The proof of this property, which will not be done here, can be derived in an analogous way as in \cite{clm}.
The detailed proof would be lengthy and tedious, so we omit it for brevity.  

\section{Numerical methods}\label{sec:method}
From the analytical results in \cite{Mauser,Najman}, the solution of the Dirac equation (\ref{nlsw}) decays very fast at infinity for localized initial data. Thus,
to give the numerical discretization, we truncate the Dirac equation (\ref{nlsw}) in the whole space to a bounded interval $(a,b)$ which is large enough, and impose periodic boundary conditions as done in \cite{Bao1,Bao2,Tang2,Tang3}, i.e.
\begin{align}
&i\pa_t \Phi^\eps=-\frac{i}{\eps}\alpha \pa_x \Phi^\eps+\frac{1}{\eps^2}\beta \Phi^\eps+\left[V_e(t)+V_m(t)\alpha\right]\Phi^\eps+\lambda\left(\beta \Phi^\eps,\Phi^\eps\right)\beta \Phi^\eps,\ x\in (a,b),\nonumber\\
&\Phi^\eps(t,a)=\Phi^\eps(t,b),\ \pa_x\Phi^\eps(t,a)=\pa_x\Phi^\eps(t,b),\ t\geq0;\quad \Phi^\eps(0,x)=\Phi_0(x).\label{bc}
\end{align}
\subsection{Uniformly accurate schemes}
Based on the two-scale formulation, we give the numerical discretization to solve the nonlinear Dirac equation.
Due to the truncation (\ref{bc}), the corresponding two-scale problem (\ref{TS}) reads
\begin{subequations}\label{nlsw trun}
\begin{align}
&\pa_t U^\eps+\frac{1}{\eps^2}\pa_\tau U^\eps=-\frac{1}{\eps}A(\tau)\pa_x U^\eps+F(t,\tau,U^\eps),\quad t>0,\ \tau\in\TT,\ x\in (a,b),\label{nlsw trun1}\\
& U^\eps(t,\tau,x)=U^\eps(t,\tau+2\pi,x),\quad t\geq0,\ \tau\in\TT,\ x\in\overline{(a,b)},\\
&U^\eps(t,\tau,a)=U^\eps(t,\tau,b),\quad t\geq0,\ \tau\in\TT,
\end{align}
\end{subequations}
with the constructed initial data
\begin{equation}\label{initial add}
U^\eps(0,\tau,x)=U^{\eps,0}(\tau,x),\quad  \tau\in\TT,\ x\in\overline{(a,b)}.
\end{equation}

Let $\Delta t>0$ be the time step, and denote $t_n=n\Delta t$ for $n=0,1,\ldots$
Let us choose mesh sizes $\Delta x=(b-a)/N$ and $\Delta\tau=2\pi/N_\tau$ with $N$ and $N_\tau$ two positive even integers. We also denote the grid points in $x$ and $\tau$ by
\begin{equation*}
x_j:=a+j\Delta x,\quad j=0,1,\ldots, N;\qquad \tau_j=j\Delta\tau,\quad j=0,1,\ldots,N_\tau.
\end{equation*}

%

For first order time discretization, inspired by the scheme in \cite{cclm}, we propose the following semi-implicit Euler method:
\begin{align}
\frac{U^{n+1}(\tau,x)-U^{n}(\tau,x)}{\Delta t}+\frac{1}{\eps^2}\partial_\tau U^{n+1}(\tau,x)=-\frac{1}{\eps}A(\tau)\partial_xU^{n+1}(\tau,x)+F^n(\tau,x),\label{semiimplicit}
\end{align}
where we denote $U^{n}(\tau,x)\approx U^{\eps}(t_n,\tau,x)$ and $F^n(\tau,x)\approx F(t_n,\tau,U^\eps(t_n,\tau,x))$, for $n=0,1,\ldots$ 
\subsubsection{Uniform accuracy of the semi-implicit scheme \eqref{semiimplicit}}
We recall that the initial data $U_3^\eps(\tau)$ given by \eqref{U3eps} has been obtained by Chapman-Enskog expansions, in order to ensure the uniform boundedness of $U^\eps$, $\partial_tU^\eps$ and $\partial_t^2U^\eps$ with respect to $\eps$. Although we omit the rigorous proof of this property, we assume that this holds and use it to prove the uniform accuracy of the semi-discrete scheme \eqref{semiimplicit} initialized with $U^0=U_3^\eps$.
In fact, we have the following result.
\begin{theorem}
\label{theo1}
Assume that the solution $U^\eps(t,\tau,x)$ of \eqref{TS} has uniformly bounded (with respect to $\eps$) derivatives for $t\in[0,T]$, up to order two. Let $U^n(\tau,x)$ be the solution of the numerical scheme \eqref{semiimplicit}, initialized with $U^0=U^{\eps,0}$, and $\Delta t>0$. Then, for $n\Delta t\leq T$, we have
$$\left\|U^\eps(n\Delta t)-U^n\right\|_{L^\infty_\tau L^2_x}\leq C\Delta t,$$
where $C$ is a constant independent of $n$, $\Delta t$ and $\eps$.
\end{theorem}
\begin{proof}
Taking the Fourier transform of \eqref{semiimplicit} with respect to the variable $x$, we get
\begin{align*}
\frac{\widehat{U^{n+1}}-\widehat{U^{n}}}{\Delta t}+\frac{1}{\eps^2}\partial_\tau \widehat{U^{n+1}}=-\frac{1}{\eps}A(\tau)i\mu \widehat{U^{n+1}}+\widehat{F^n},
\end{align*}
where $\mu$ is the Fourier variable associated to $x$. In terms of the following quantity
$$Q(V)=V+\frac{\Delta t}{\eps}A(\tau)i\mu V+\frac{\Delta t}{\eps^2}\partial_\tau V,$$
the scheme reads
\begin{equation}\label{eqQ}Q\left(\widehat {U^{n+1}}\right)=\widehat{U^{n}}+\Delta t \widehat{F^n}.\end{equation}

\bigskip
\noindent{\em Step 1: invertibility property of $Q$.} We claim that $Q$ is invertible on $L^\infty_\tau L^2_\mu$ and that
\begin{equation}\label{proof add1}
\|Q^{-1}(W)\|_{L^\infty_\tau L^2_\mu}\leq \|W\|_{L^\infty_\tau L^2_\mu}.
\end{equation}
Indeed, consider the equation $Q(V)=W$ and set
$$Z(\tau):=\begin{pmatrix}\fe^{-i\tau}&0\\ 0&\fe^{i\tau}\end{pmatrix}V.$$
We get the equation
$$Z+i\Omega Z+\frac{\Delta t}{\eps^2}\partial_\tau Z=\begin{pmatrix}\fe^{-i\tau}&0\\ 0&\fe^{i\tau}\end{pmatrix}W,$$
with $$\Omega=\begin{pmatrix}\Delta t/\eps^2&\mu \Delta t/\eps\\ \mu \Delta t/\eps&-\Delta t/\eps^2\end{pmatrix}.$$
Then,
$$\frac{d}{d\tau}\left(\fe^{\frac{\eps^2}{\Delta t}(I+i\Omega)\tau}Z\right)=\frac{\Delta t}{\eps^2}\fe^{\frac{\eps^2}{\Delta t}(I+i\Omega)\tau}\begin{pmatrix}\fe^{-i\tau}&0\\ 0&\fe^{i\tau}\end{pmatrix}W.$$
We integrate this expression between $\tau$ and $\tau+2\pi$ and use the periodicity of $Z(\tau)$ to obtain
\begin{align*}&\left(\fe^{\frac{\tau+2\pi}{\Delta t}\eps^2}\fe^{i\frac{\tau+2\pi}{\Delta t}\eps^2\Omega}-\fe^{\frac{\tau}{\Delta t}\eps^2}\fe^{i\frac{\tau}{\Delta t}\eps^2\Omega}\right)Z(\tau)\\
&\hspace*{3cm}=\frac{\Delta t}{\eps^2}\int_\tau^{\tau+2\pi}\fe^{\frac{s}{\Delta t}\eps^2}\fe^{i\frac{s}{\Delta t}\eps^2\Omega}\begin{pmatrix}\fe^{-is}&0\\ 0&\fe^{is}\end{pmatrix}W(s)ds.\end{align*}
\sloppy Using that $\Omega$ is real-valued and symmetric, we get, with the notation $\|Z\|=\sqrt{|Z_1|^2+|Z_2|^2}$ for all $Z\in \CC^2$, that
\begin{align*}&\left\|\fe^{\frac{\tau+2\pi}{\Delta t}\eps^2}\fe^{i\frac{\tau+2\pi}{\Delta t}\eps^2\Omega}Z(\tau)-\fe^{\frac{\tau}{\Delta t}\eps^2}\fe^{i\frac{\tau}{\Delta t}\eps^2\Omega}Z(\tau)\right\|\\
&\hspace*{3cm}\geq \fe^{\frac{\tau+2\pi}{\Delta t}\eps^2}\left\|\fe^{i\frac{\tau+2\pi}{\Delta t}\eps^2\Omega}Z(\tau)\right\|-\fe^{\frac{\tau}{\Delta t}\eps^2}\left\|e^{i\frac{\tau}{\Delta t}\eps^2\Omega}Z(\tau)\right\|\\
&\hspace*{3cm}= \left(\fe^{\frac{\tau+2\pi}{\Delta t}\eps^2}-\fe^{\frac{\tau}{\Delta t}\eps^2}\right)\left\|Z(\tau)\right\|,
\end{align*}
and that
\begin{align*}&\left\|\frac{\Delta t}{\eps^2}\int_\tau^{\tau+2\pi}\fe^{\frac{s}{\Delta t}\eps^2}\fe^{i\frac{s}{\Delta t}\eps^2\Omega}\begin{pmatrix}\fe^{-is}&0\\ 0&\fe^{is}\end{pmatrix}W(s)ds\right\|_{L^2\mu}\\
&\hspace*{3cm}\leq \frac{\Delta t}{\eps^2}\int_\tau^{\tau+2\pi}\fe^{\frac{s}{\Delta t}\eps^2}\left\|\begin{pmatrix}\fe^{-is}&0\\ 0&\fe^{is}\end{pmatrix}W(s)\right\|_{L^2\mu}\ds\\
&\hspace*{3cm}\leq \left(\fe^{\frac{\tau+2\pi}{\Delta t}\eps^2}-\fe^{\frac{\tau}{\Delta t}\eps^2}\right)\|W\|_{L^\infty_\tau L^2_\mu}.
\end{align*}
These two inequalities give $\|Z(\tau)\|_{L^2_\mu}\leq \|W\|_{L^\infty_\tau L^2_\mu}$, which proves the claim.
\bigskip

\bigskip
\noindent{\em Step 2: Taylor expansion of the exact solution.} 
The exact solution of \eqref{TS} satisfies
$$U^\eps(t^n)-U^\eps(t^{n+1})+\Delta t \partial_t U^\eps(t^{n+1})=R^n\qquad \mbox{with}\quad R^n=\int_{t^n}^{t^{n+1}}(s-t^n)\partial^2_tU^\eps(s)ds$$
which yields
$$Q(\widehat{U^\eps(t^{n+1})})=\widehat{U^\eps(t^n)}+\Delta t \widehat{F}(t^{n+1},\tau,\widehat{U^\eps(t^{n+1})})-\widehat{R^n}.$$
Here we have denoted $\widehat G (\widehat U)=\widehat{G(U)}$.

We now introduce the error $E^n=\widehat{U^\eps(t^n)}-\widehat{U^n}$ and subtracting \eqref{eqQ} from this equation, we get
\begin{equation}\label{Eestimate}Q(E^{n+1})=E^n+\Delta t\left(\widehat{F}(t^{n+1},\tau,\widehat{U^\eps(t^{n+1})})-\widehat{F}(t^{n},\tau,\widehat{U^n})\right)-\widehat{R^n}.\end{equation}

\bigskip
\noindent{\em Step 3: error estimate.} We observe from the expression \eqref{F def} of $F$ and the assumptions on $V_e$ and $V_m$ that $F$ is locally Lipschitz continuous with respect to $t$ and $U^\eps$.
Let us fix $M>0$ such that $\|U^\eps\|_{L^\infty_tL^\infty_\tau L^2_x}\leq M$ for $t\in[0,T]$. Let $N$ be the number of time discretization points and $\Delta t=T/N$.
Let $n_0$ be the largest integer $n\leq N$ such that $\|U^n\|_{L^\infty_\tau L^2_x}\leq 2M$.
Our aim is to prove that $n_0=N$ and to estimate the error.

Let us proceed by contradiction and assume that  $n_0\leq N-1$. Then we have
$$\left\|\widehat{F}(t^{n+1},\tau,\widehat{U^\eps(t^{n+1})})-\widehat{F}(t^{n},\tau,\widehat{U^n})\right\|_{L^\infty_\tau L^2_\mu}\leq C\left(\Delta t+\|E^n\|_{L^\infty_\tau L^2_\mu}\right),$$
where $C$, here and after, is a generic constant which only depends on $M$ and $T$.
Applying \eqref{proof add1} to \eqref{Eestimate}, we deduce that
$$\|E^{n+1}\|_{L^\infty_\tau L^2_\mu}\leq (1+C\Delta t)\|E^n\|_{L^\infty_\tau L^2_\mu}+C\Delta t^2+\|R^n\|_{L^\infty_\tau L^2_\mu}.$$
Now, in order to estimate $R^n$, we apply the assumption made in Theorem \ref{theo1} saying that $\partial_t^2U$ is uniformly bounded with respect to $\eps$. Thus $\|R^n\|_{L^\infty_\tau L^2_\mu}\leq C\Delta t^2$. Therefore
$$\|E^{n+1}\|_{L^\infty_\tau L^2_\mu}\leq (1+C\Delta t)\|E^n\|_{L^\infty_\tau L^2_\mu}+C\Delta t^2,$$
for all $n\leq n_0$,
which implies
$$\|E^{n+1}\|_{L^\infty_\tau L^2_\mu}\leq C\Delta t.$$
In particular, if $n_0\Delta t<T$ and if we choose $C\Delta t<M$ then we deduce from this estimate that
$$\|\widehat {U^{n_0+1}}\|_{L^\infty_\tau L^2_\mu}\leq \|\widehat{U^\eps(t^{n_0+1})}\|_{L^\infty_\tau L^2_\mu}+\|E^{n_0+1}\|_{L^\infty_\tau L^2_\mu}\leq 2M.$$
This contradicts the fact that $n_0$ is the largest integer $n\leq N$ such that $\|U^n\|_{L^\infty_\tau L^2_x}\leq 2M$. We conclude that $n_0=N$. It is also clear from the above estimates that we have
$$\|E^{n}\|_{L^\infty_\tau L^2_\mu}\leq C\Delta t,$$
for all $n\leq N$, which concludes the proof.
\end{proof}

\subsubsection{The fully discretized scheme} 
In the $x$-direction, we apply the Fourier transform and obtain for $l=-N/2,\ldots,N/2-1$,
\begin{align}\label{semi-diff}
\frac{\widehat{U^{n+1}_l}(\tau)-\widehat{U^{n}_l}(\tau)}{\Delta t}+\frac{1}{\eps^2}\partial_\tau \widehat{U^{n+1}_l}(\tau)=-\frac{1}{\eps}A(\tau)i\mu_l\widehat{U^{n+1}_l}(\tau)+\widehat{F^n_l}(\tau),
\end{align}
where $\mu_l=\frac{2\pi l}{b-a}$ and
$$U^n(\tau,x)=\sum_{l=-N/2}^{N/2-1}\widehat{U^{n}_l}(\tau)\fe^{i\mu_l(x-a)},\quad F^n(\tau,x)=\sum_{l=-N/2}^{N/2-1}\widehat{F^n_l}(\tau)\fe^{i\mu_l(x-a)},$$
with
$$\widehat{U^{n}_l}(\tau)=\frac{1}{N}\sum_{j=0}^{N-1}U^n(\tau,x_j)\fe^{i\mu_l(x_j-a)},\quad
\widehat{F^{n}_l}(\tau)=\frac{1}{N}\sum_{j=0}^{N-1}F^n(\tau,x_j)\fe^{i\mu_l(x_j-a)}.$$
In terms of the two components $U^\eps=(u_1,u_2)^T$ and $F=(f_1,f_2)^T$, the numerical scheme (\ref{semi-diff}) reads
\begin{subequations}\label{semi-diff2}
\begin{align}
&\frac{\widehat{(u^{n+1}_1)}_l(\tau)-\widehat{(u^{n}_1)}_l(\tau)}{\Delta t}+\frac{1}{\eps^2}\partial_\tau \widehat{(u^{n+1}_1)}_l(\tau)=-\frac{ i\mu_l}{\eps}\fe^{2i\tau}\widehat{(u^{n+1}_2)}_l(\tau)+\widehat{(f^n_1)}_l(\tau),\\
&\frac{\widehat{(u^{n+1}_2)}_l(\tau)-\widehat{(u^{n}_2)}_l(\tau)}{\Delta t}+\frac{1}{\eps^2}\partial_\tau \widehat{(u^{n+1}_2)}_l(\tau)=-\frac{ i\mu_l}{\eps}\fe^{-2i\tau}\widehat{(u^{n+1}_1)}_l(\tau)+\widehat{(f^n_2)}_l(\tau).
\end{align}
\end{subequations}
Then we discretize the $\tau$-direction as follows
\begin{align*}
&\widehat{(u^{n}_j)}_l=\left(\widehat{(u^{n}_j)}_l(\tau_0),\ldots,\widehat{(u^{n}_j)}_l(\tau_{N_\tau-1})\right)^T,\quad l=-N/2,\ldots,N/2-1,\\
&\widehat{(f^n_j)}_l=\left(\widehat{(f^n_j)}_l(\tau_0),\ldots,\widehat{(f^n_j)}_l(\tau_{N_\tau-1})\right)^T,\quad j=1,2.
\end{align*}
Introducing the Fourier pseudo-differential matrix $D_\tau$ \cite{Shen}: $$D_\tau=\left(d_{n,m}\right)\in\CC^{N_\tau\times N_\tau}, \ d_{n,m}=\frac{i}{N_\tau}\sum_{l=-N_\tau/2}^{N_\tau/2-1}l\fe^{il(\tau_n-\tau_m)},\quad n,m=1,\ldots,N_{\tau},$$
we derive the full discretization of (\ref{semi-diff2}) as
\begin{subequations}\label{full-diff}
\begin{align}
&\frac{\widehat{(u^{n+1}_1)}_l-\widehat{(u^{n}_1)}_l}{\Delta t}+\frac{1}{\eps^2}D_\tau \widehat{(u^{n+1}_1)}_l=-\frac{ i\mu_l}{\eps}\fe^{2i\tau}\widehat{(u^{n+1}_2)}_l+\widehat{(f^n_1)}_l,\\
&\frac{\widehat{(u^{n+1}_2)}_l-\widehat{(u^{n}_2)}_l}{\Delta t}+\frac{1}{\eps^2}D_\tau \widehat{(u^{n+1}_2)}_l=-\frac{ i\mu_l}{\eps}\fe^{-2i\tau}\widehat{(u^{n+1}_1)}_l+\widehat{(f^n_2)}_l,\ n\geq0,
\end{align}
\end{subequations}
for $l=-\frac{N}{2},\ldots,\frac{N}{2}-1$, where here and after $\fe^{\pm2i\tau}$ are interpreted as diagonal matrices. Solving (\ref{full-diff}) and choosing $(u_1^0,u_2^0)^T=U^{\eps,0}$, we get the detailed scheme of the first order method (UA1) for $n=0,1,\ldots,$ and $l=-N/2,\ldots,N/2-1$,
\begin{subequations}\label{ua1}
  \begin{align}
  &A_\tau\widehat{(u_1^{n+1})}_0=\frac{1}{\Delta t}\widehat{(u_1^{n})}_0+\widehat{(f_1^{n})}_0,\ A_\tau\widehat{(u_2^{n+1})}_0=\frac{1}{\Delta t}\widehat{(u_2^{n})}_0+\widehat{(f_2^{n})}_0,\\
  &B_\tau^l\widehat{(u_1^{n+1})}_l=A_\tau^l\left[\frac{1}{\Delta t}\widehat{(u_1^{n})}_l+\widehat{(f_1^{n})}_l\right]-\frac{1}{\Delta t}\widehat{(u_2^{n})}_l-\widehat{(f_2^{n})}_l,\quad l\neq0,\\
  &\widehat{(u_2^{n+1})}_l=\frac{\eps\fe^{-2i\tau}}{i\mu_l}\left[-A_\tau\widehat{(u_1^{n+1})}_l+\frac{1}{\Delta t}\widehat{(u_1^{n})}_l+\widehat{(f_1^{n})}_l\right],\quad l\neq0,
  \end{align}
\end{subequations}
where
$$A_\tau=\frac{Id}{\Delta t}+\frac{D_\tau}{\eps^2},\ A_\tau^l=\frac{\eps A_\tau\fe^{-2i\tau}}{i\mu_l},\ B_\tau^l=
A_\tau^lA_\tau-\frac{i\mu_l\fe^{-2i\tau}}{\eps},\quad l\neq0,$$
and where $Id$ denotes the identity matrix.
The inverse of the above matrices $A_\tau,B_\tau^l$ can be computed numerically once for all, so the UA1 method (\ref{ua1}) is explicit in practise. Also noticing that $B_\tau^l=-B_\tau^{-l}$, we only need to store and compute the inverse of $B_\tau^l$ for $l=-N/2,\ldots,-1$. The computational cost of the UA1 scheme is $O(NN_\tau^2\log(N))$ per time step.
\subsubsection{The uniformly accurate second order scheme}
Again inspired by \cite{cclm}, our second order method starts with a prediction step from the first order method
\begin{align*}
\frac{U^{n+1/2}(\tau,x)-U^{n}(\tau,x)}{\Delta t}+\frac{1}{\eps^2}\partial_\tau U^{n+1/2}(\tau,x)=&-\frac{1}{\eps}A(\tau)\partial_xU^{n+1/2}(\tau,x)\\
&+F^n(\tau,x),
\end{align*}
followed by the correction step
\begin{align*}
&\frac{U^{n+1}(\tau,x)-U^{n}(\tau,x)}{2\Delta t}+\frac{1}{2\eps^2}\left[\partial_\tau U^{n+1}(\tau,x)+\partial_\tau U^{n}(\tau,x)\right]\\
&=-\frac{1}{2\eps}A(\tau)\left[\partial_xU^{n+1}(\tau,x)+\partial_xU^{n}(\tau,x)\right]+F^{n+1/2}(\tau,x),
\end{align*}
where $F^{n+1/2}(\tau,x):=F(t_{n+1/2},\tau,U^{n+1/2}(\tau,x))$. Then by applying the Fourier pseudo-spectral discretization similarly as before, we end up with the following detailed second order scheme (UA2). Choosing $(u_1^0,u_2^0)^T=U^{\eps,0}$, then for $n=0,1,\ldots,$ and $l=-N/2,\ldots,$ $N/2-1$, 
\begin{subequations}\label{ua2}
  \begin{align}
  &A_\tau^+\widehat{(u_1^{n+1})}_0=A_\tau^-\widehat{(u_1^{n})}_0+\widehat{(f_1^{n+1/2})}_0,\ A_\tau^+\widehat{(u_2^{n+1})}_0=A_\tau^-\widehat{(u_2^{n})}_0+\widehat{(f_2^{n+1/2})}_0,\\
  &B_\tau^{+,l}\widehat{(u_1^{n+1})}_l=A_\tau^{+,l}\left[A_\tau^-\widehat{(u_1^{n})}_l+\widehat{(f_1^{n+1/2})}_l\right]-
  A_\tau^-\widehat{(u_2^{n})}_l-\widehat{(f_2^{n+1/2})}_l\nonumber\\
  &\qquad\qquad\qquad+\frac{i\mu_l\fe^{-2i\tau}}{2\eps}\widehat{(u_1^n)}_l-A_\tau^+\widehat{(u_2^n)}_l,\quad l\neq0,\\
  &\widehat{(u_2^{n+1})}_l=-\widehat{(u_2^n)}_l+\frac{2\eps\fe^{-2i\tau}}{i\mu_l}\left[A_\tau^-\widehat{(u_1^{n})}_l-A_\tau^+\widehat{(u_1^{n+1})}_l
  +\widehat{(f_1^{n+1/2})}_l\right],\quad l\neq0,
  \end{align}
\end{subequations}
with
\begin{align*}
  &A_\tau^{1/2}\widehat{(u_1^{n+1/2})}_0=\frac{1}{2\Delta t}\widehat{(u_1^{n})}_0+\widehat{(f_1^{n})}_0,\ A_\tau^{1/2}\widehat{(u_2^{n+1/2})}_0=\frac{1}{2\Delta t}\widehat{(u_2^{n})}_0+\widehat{(f_2^{n})}_0,\\
  &B_\tau^{1/2,l}\widehat{(u_1^{n+1/2})}_l=A_\tau^{1/2,l}\left[\frac{1}{2\Delta t}\widehat{(u_1^{n})}_l+\widehat{(f_1^{n})}_l\right]-\frac{1}{2\Delta t}\widehat{(u_2^{n})}_l-\widehat{(f_2^{n})}_l,\quad l\neq0,\\
  &\widehat{(u_2^{n+1/2})}_l=\frac{\eps\fe^{-2i\tau}}{i\mu_l}\left[-A_\tau^{1/2}\widehat{(u_1^{n+1/2})}_l+\frac{1}{2\Delta t}\widehat{(u_1^{n})}_l+\widehat{(f_1^{n})}_l\right],\quad l\neq0,
  \end{align*}
where
\begin{align*}
&A_\tau^\pm=\frac{Id}{\Delta t}\pm\frac{D_\tau}{2\eps^2},\ A_\tau^{+,l}=\frac{2\eps A_\tau^+\fe^{-2i\tau}}{i\mu_l},\ B_\tau^{+,l}=
A_\tau^{+,l}A_\tau^+-\frac{i\mu_l\fe^{-2i\tau}}{2\eps},\quad l\neq0,\\
&A_\tau^{1/2}=\frac{Id}{2\Delta t}+\frac{D_\tau}{\eps^2},\ A_\tau^{1/2,l}=\frac{\eps A_\tau^{1/2}\fe^{-2i\tau}}{i\mu_l},\ B_\tau^{1/2,l}=
A_\tau^{1/2,l}A_\tau^{1/2}-\frac{i\mu_l\fe^{-2i\tau}}{\eps}.
\end{align*}
Similarly, we remark that the inverse of matrices $A_\tau^+,A_\tau^{1/2},B_\tau^{+,l},B_\tau^{1/2,l}$ can be computed once for all, and by noticing that $B_\tau^{1/2,l}=-B_\tau^{1/2,-l}$ and $B_\tau^{+,l}=-B_\tau^{+,-l}$, we can only store the inverse of $B_\tau^{+,l}$ and $B_\tau^{1/2,l}$ for $l=-N/2,\ldots,-1$. The computational cost of the UA2 scheme is also $O(NN_\tau^2\log(N))$ per time step. 

Thanks to the semi-implicit approximations in time, the time step $\Delta t$ of UA1 and UA2 methods are free from any CFL-type conditions on $\Delta\tau$, $\Delta x$ or stability condition on $\eps$. After obtaining $U^{n}(\tau,x)$, we take $\tau=t_n/\eps^2$ and consider the inverse of (\ref{filter}) to get the approximation of $\Phi^\eps(t_n,x)$, i.e. $$\Phi^\eps(t_n,x)\approx \begin{pmatrix}\fe^{-it_n/\eps^2}&0\\ 0&\fe^{it_n/\eps^2}\end{pmatrix}U^{n}(t_n/\eps^2,x).$$

It is clear that the finite difference time discretisation used in UA2 gives formally second order approximation with error terms depending on $\partial_t^3U^\eps(t,\tau,x)$. Thus, with the prepared $U^\eps_5$ as the initial data which bounds $ \partial_t^3U^\eps$ for all $0<\eps\leq1$, a global uniform second order temporal error bound could be established at a finite time for UA2 similarly as the Theorem \ref{theo1}. We omit the detailed proof here and will illustrate this property by the coming numerical experiments.

\section{Numerical results}\label{sec:result}
For comparison purpose, we define the initial data
\begin{align*}
U_0^\eps(\tau,x):=\Phi_0(x),
\end{align*}
along with $U_1^\eps,U_2^\eps,U_3^\eps,U_4^\eps$ and $U_5^\eps$. We shall test the error of the proposed UA1 scheme (\ref{ua1}) and the UA2 scheme (\ref{ua2}) under those choices of initial data for (\ref{initial add}) by performing the following three numerical experiments. The computational domain is chosen as $(a,b)=(-8,8)$. The `exact' solution is obtained numerically by the UA2 scheme with very small step sizes, e.g. $\Delta t=10^{-6},\Delta x=1/64,N_\tau=64.$ We shall focus on the time discretization error.

\emph{Example I: (nonlinear without magnetic potential)}
We take the potential and initial data in (\ref{nlsw}) as
$$V_e=\frac{1-x}{2+2x^2},\quad V_m=0,\quad \phi_1^\eps(t=0)=\frac{\fe^{-x^2}}{\sqrt{2}},\quad \phi_2^\eps(t=0)=\fe^{-\sqrt{2}x^2},$$
and choose $\lambda=0.5$. We solve the Dirac equation (\ref{nlsw trun}) till $t=0.5$, and show the error of the numerical solution $(\phi_1^n,\phi_2^n)$ with $n=t/\Delta t$ as
$$\|\phi_1^\eps(t)-\phi_1^n\|_{l^\infty}+\|\phi_2(t)-\phi_2^n\|_{l^\infty}.$$
The time discretization errors of the UA1 and UA2 methods under different $\eps$ are shown in Fig. \ref{fig:err1} and Fig. \ref{fig:err2}, respectively, where we take the spatial mesh size of the numerical method small enough, e.g. $\Delta x=1/64,N_\tau=32$. 
\begin{figure}[h]
$$
\begin{array}{cc}
\psfig{figure=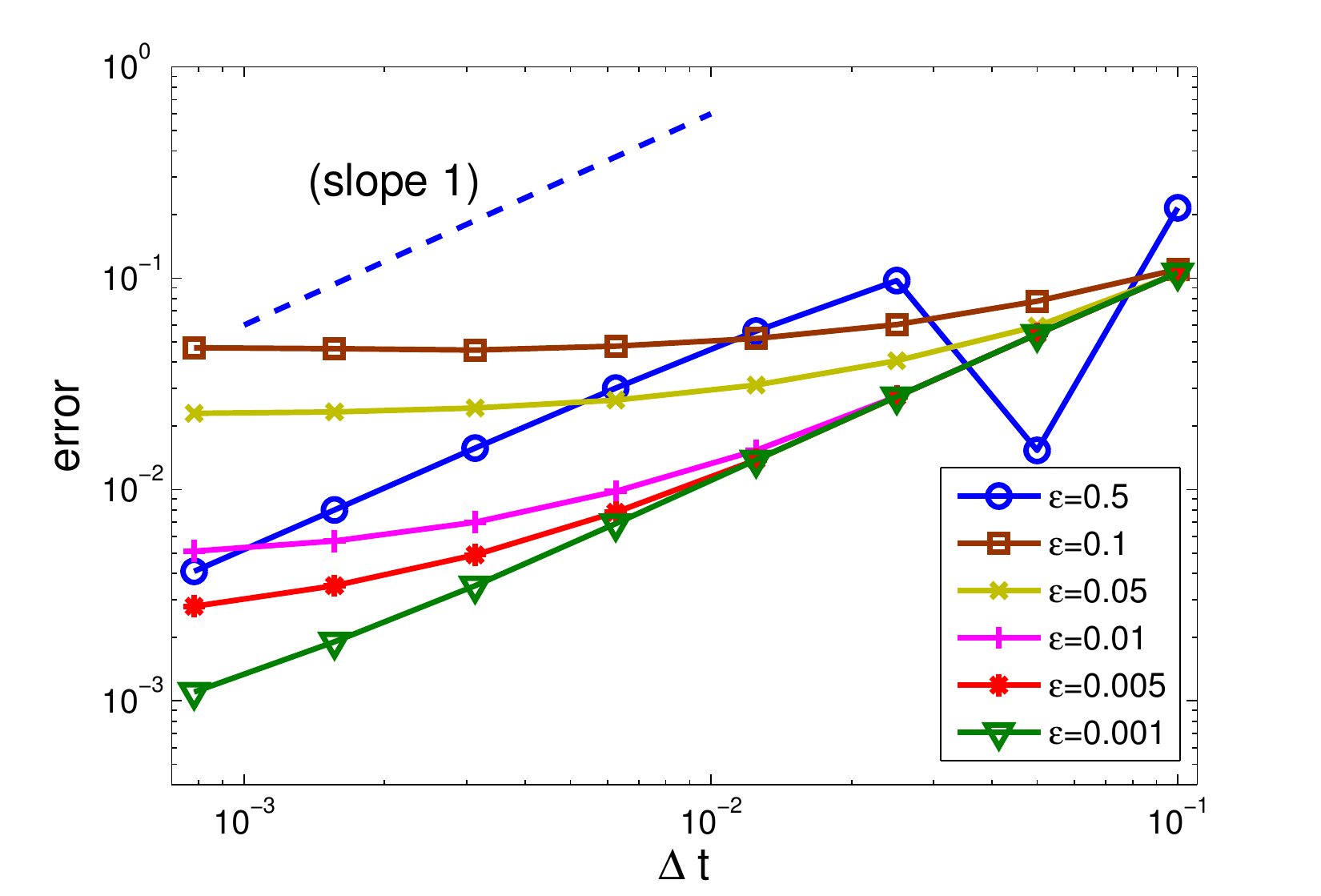,height=4.3cm,width=5.5cm}&\psfig{figure=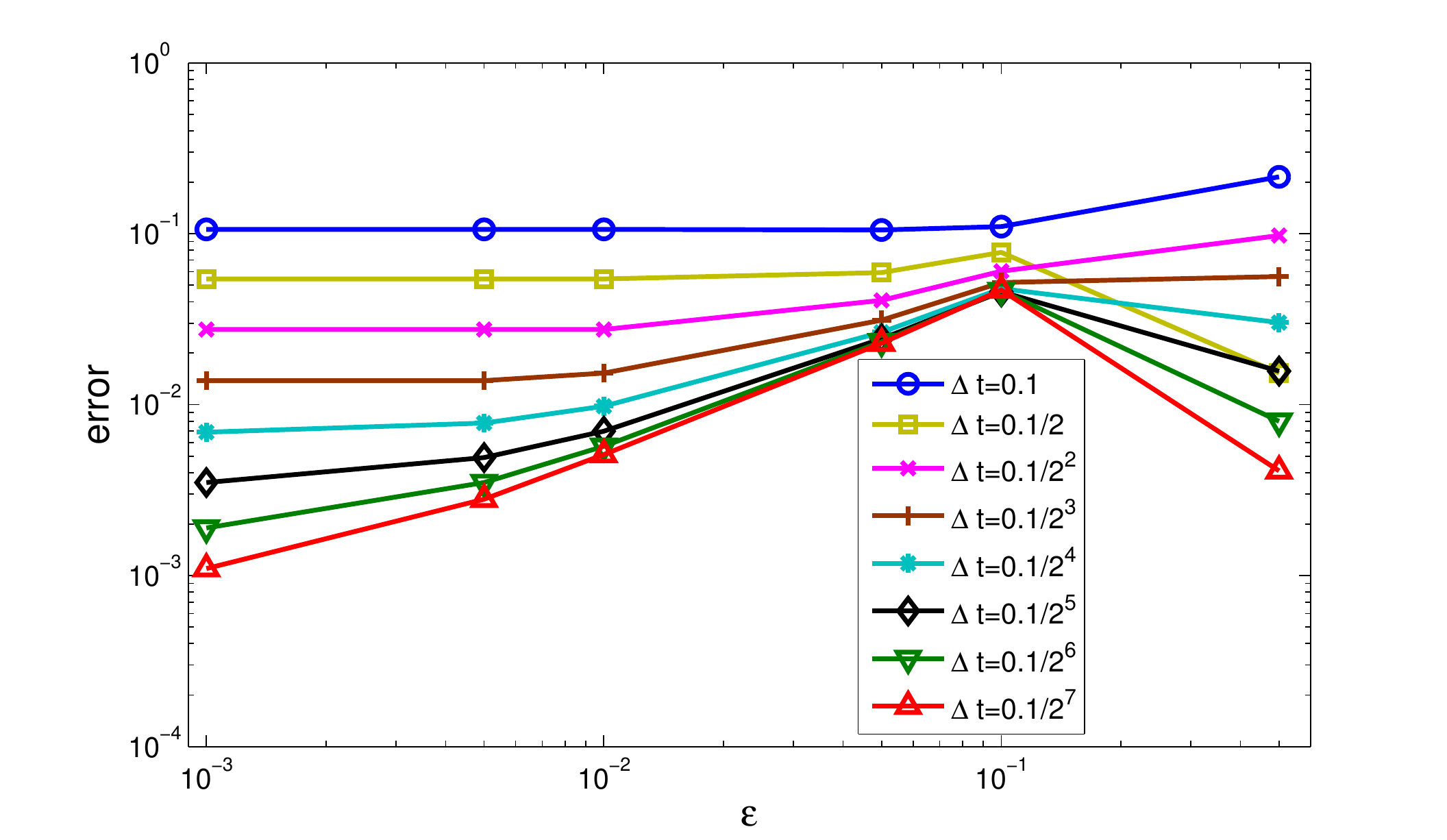,height=4.3cm,width=5.5cm}\\
\psfig{figure=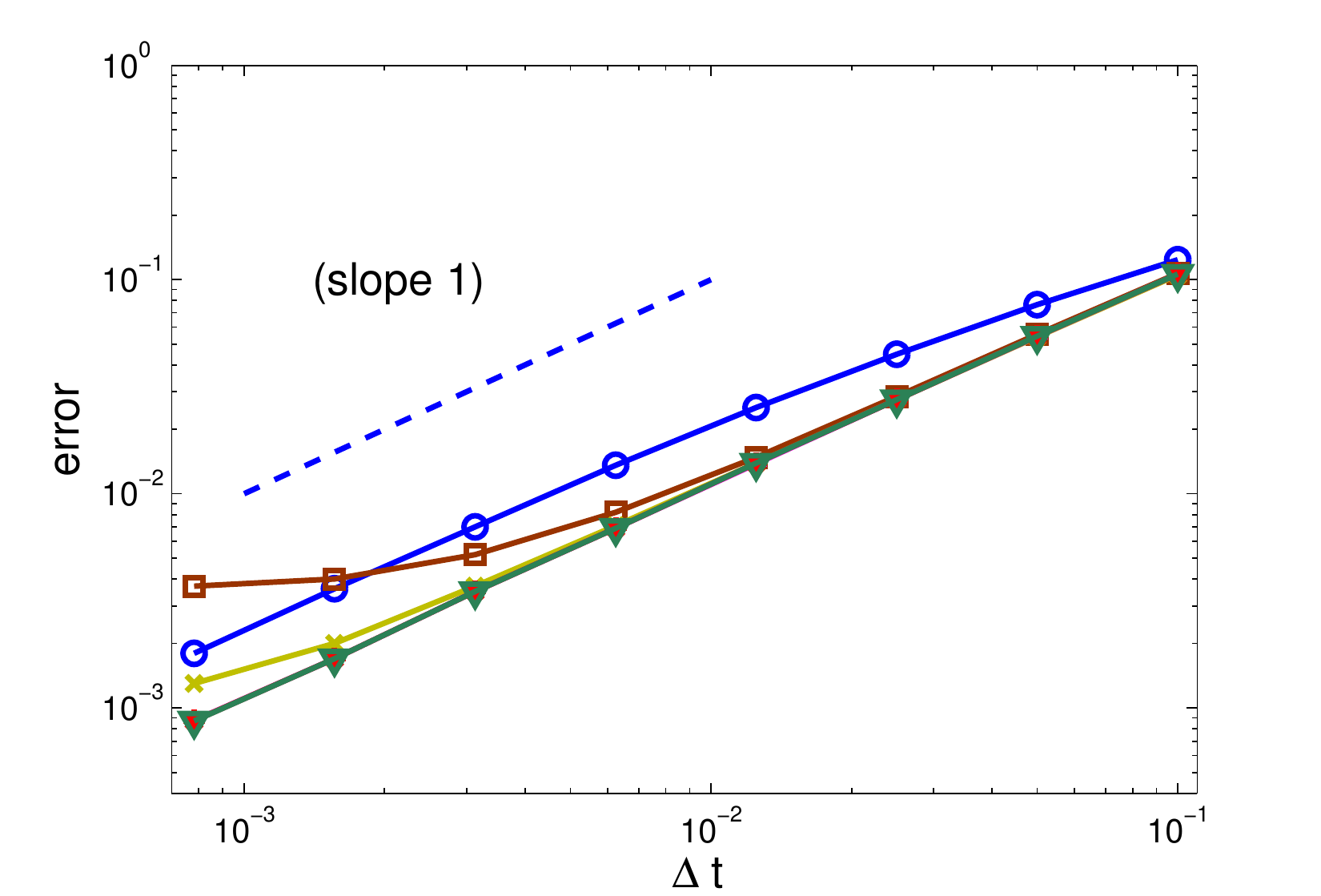,height=4.3cm,width=5.5cm}&\psfig{figure=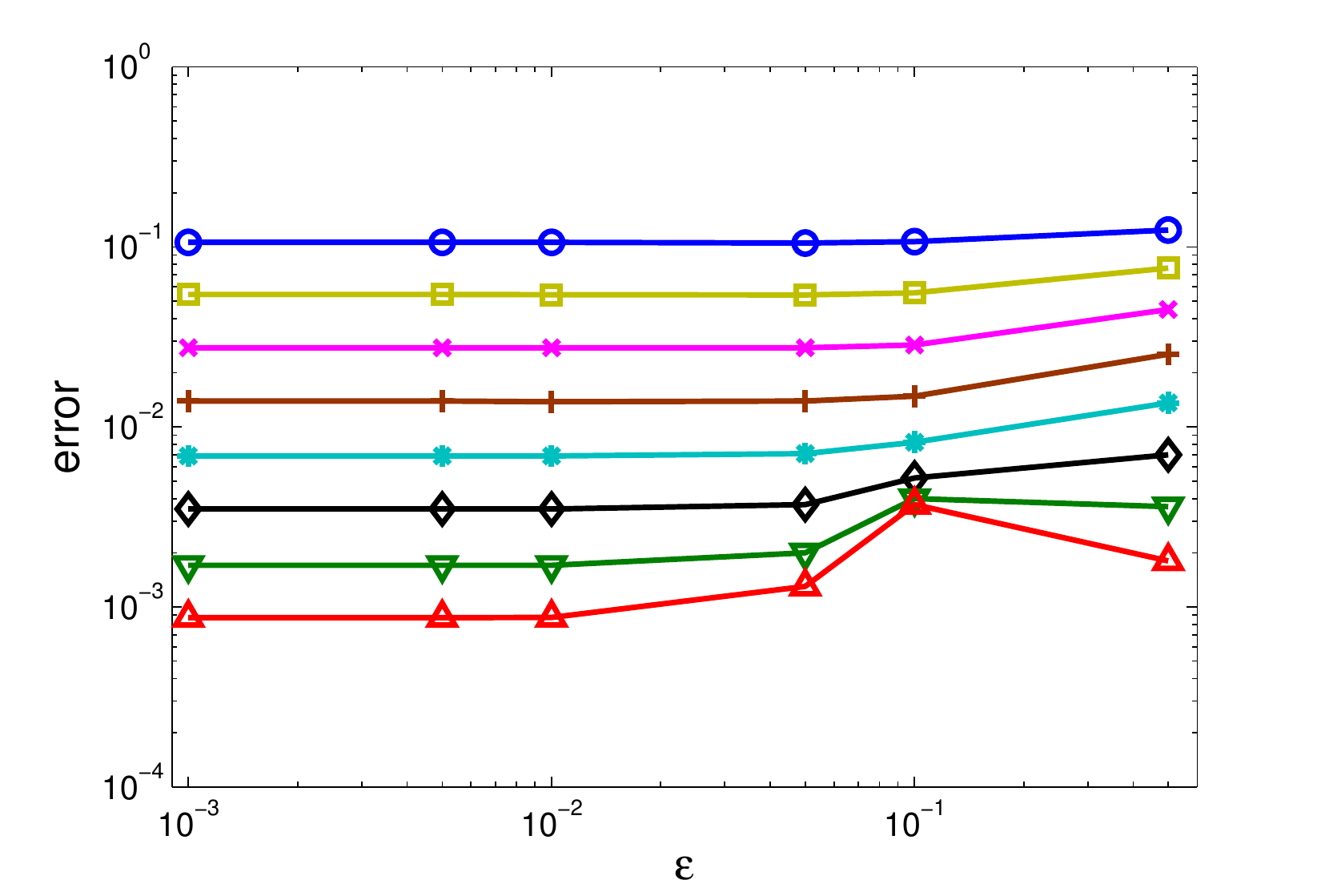,height=4.3cm,width=5.5cm}\\
\psfig{figure=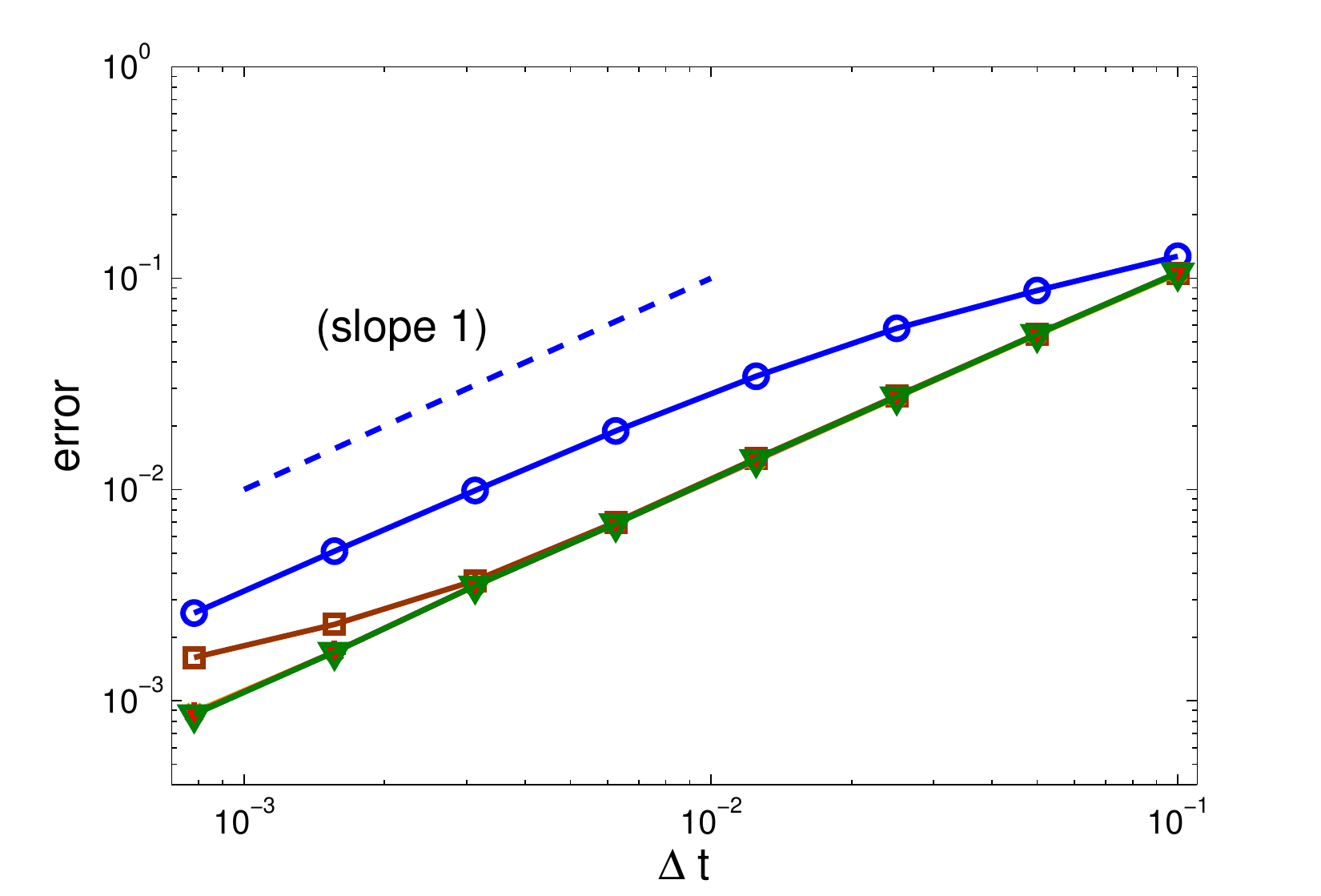,height=4.3cm,width=5.5cm}&\psfig{figure=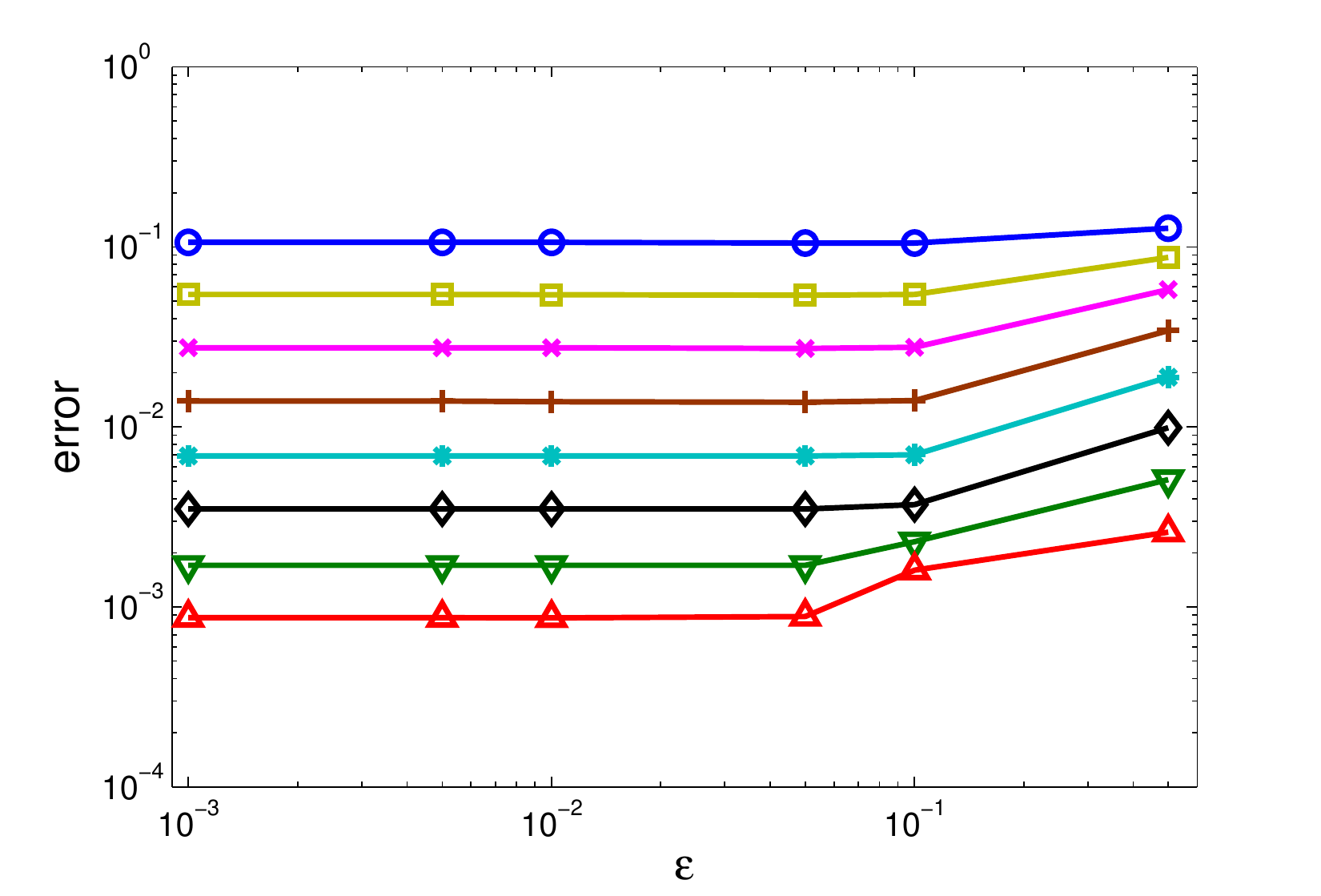,height=4.3cm,width=5.5cm}\\
\psfig{figure=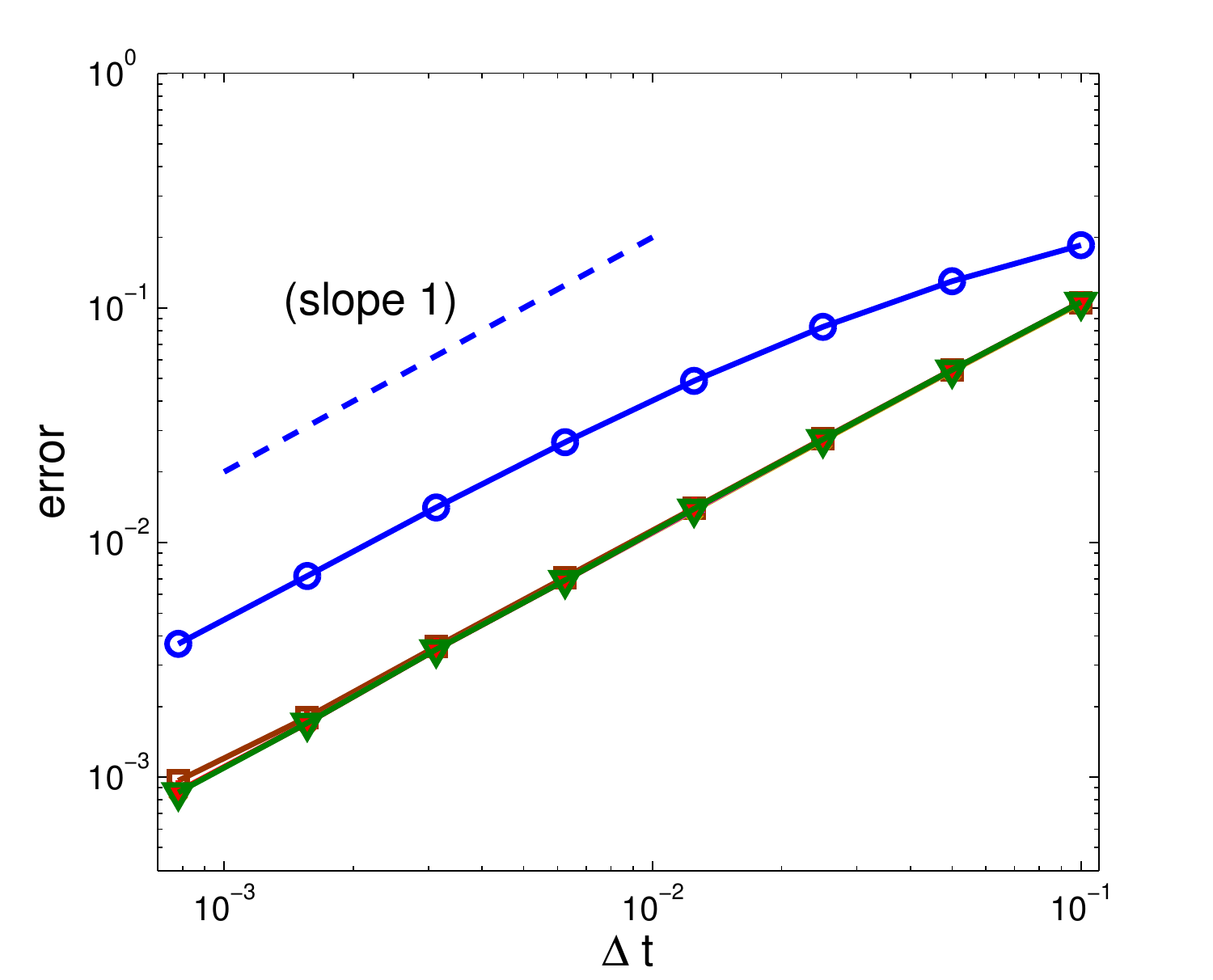,height=4.3cm,width=5.5cm}&\psfig{figure=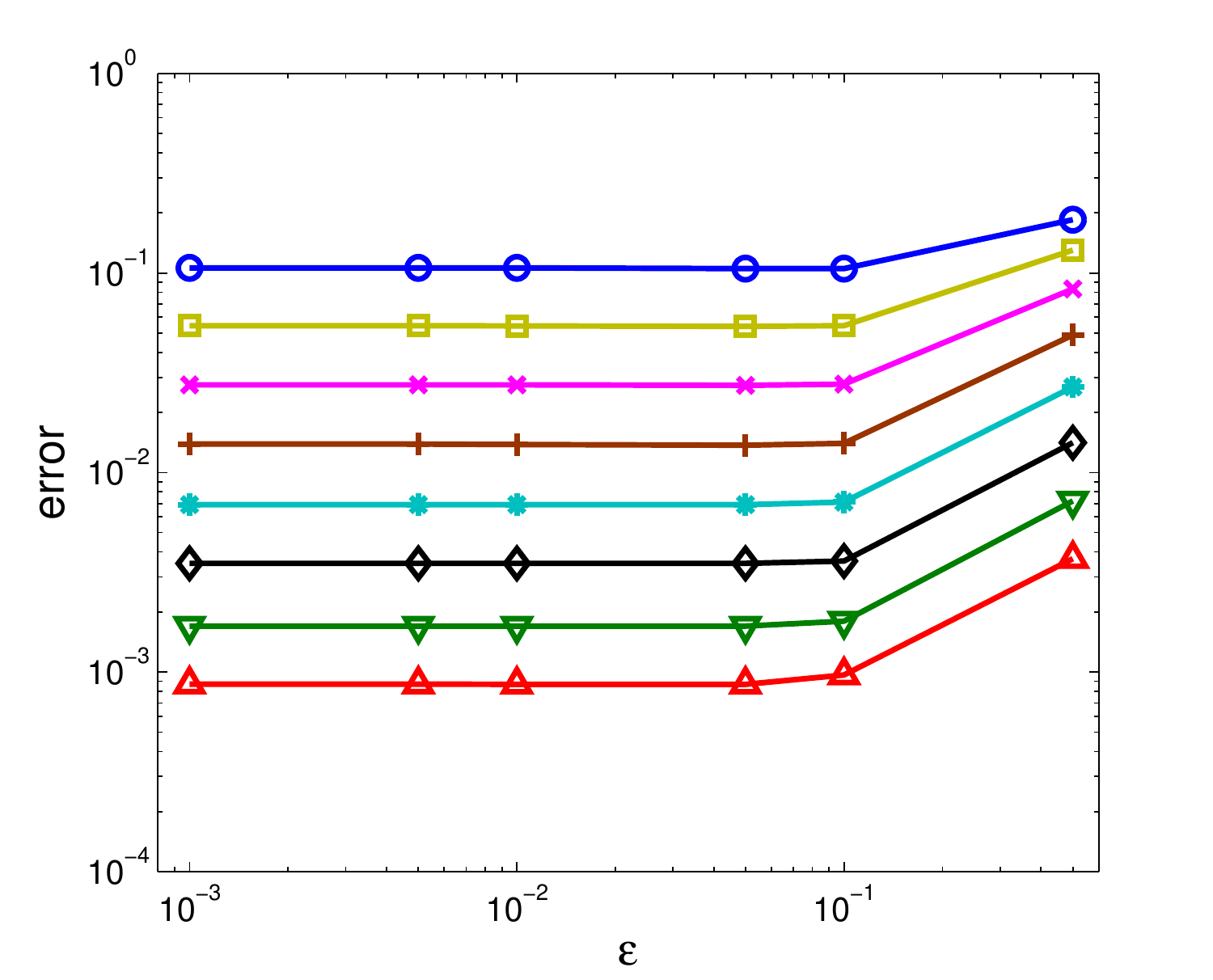,height=4.3cm,width=5.5cm}
\end{array}
$$
\caption{Temporal error of the UA1 method in Example I with respect to $\Delta t$ and $\eps$: results of using $U_0^\eps$ (first row); results of using $U_1^\eps$ (second row) results of using $U_2^\eps$ (third row);  results of using $U_3^\eps$ (last row).}\label{fig:err1}
\end{figure}

\begin{figure}[h!]
$$
\begin{array}{cc}
\psfig{figure=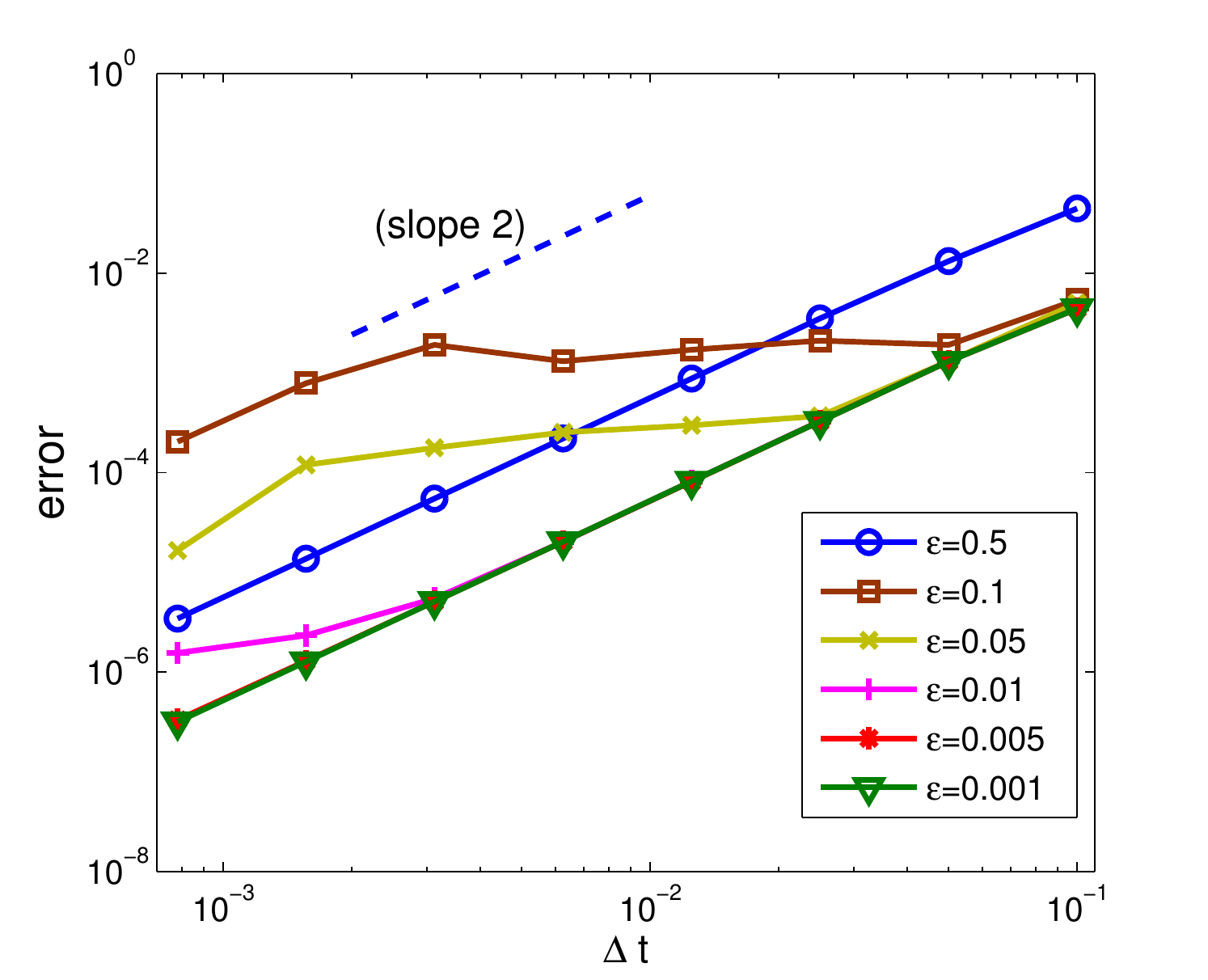,height=4.3cm,width=5.5cm}&\psfig{figure=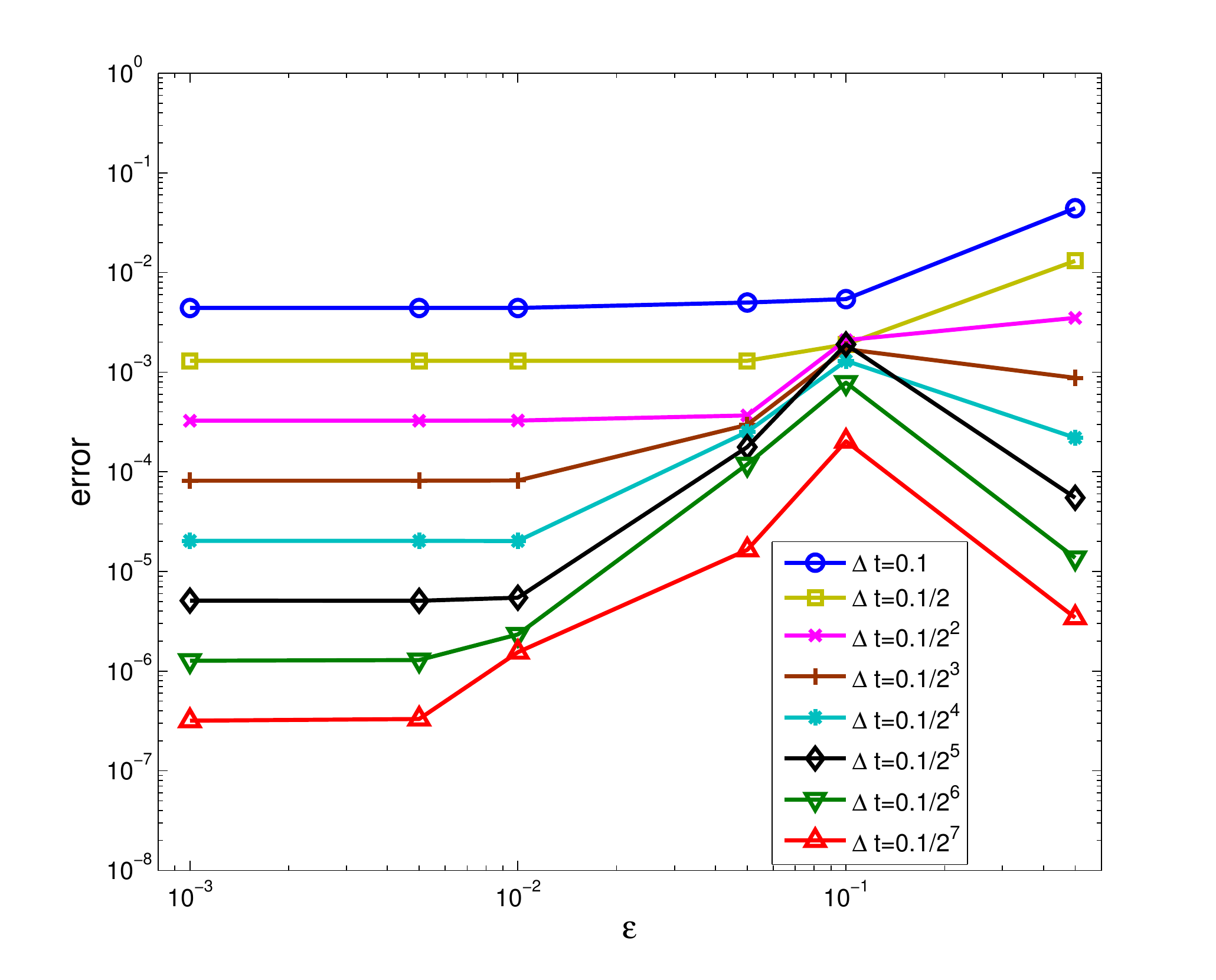,height=4.3cm,width=5.5cm}\\
\psfig{figure=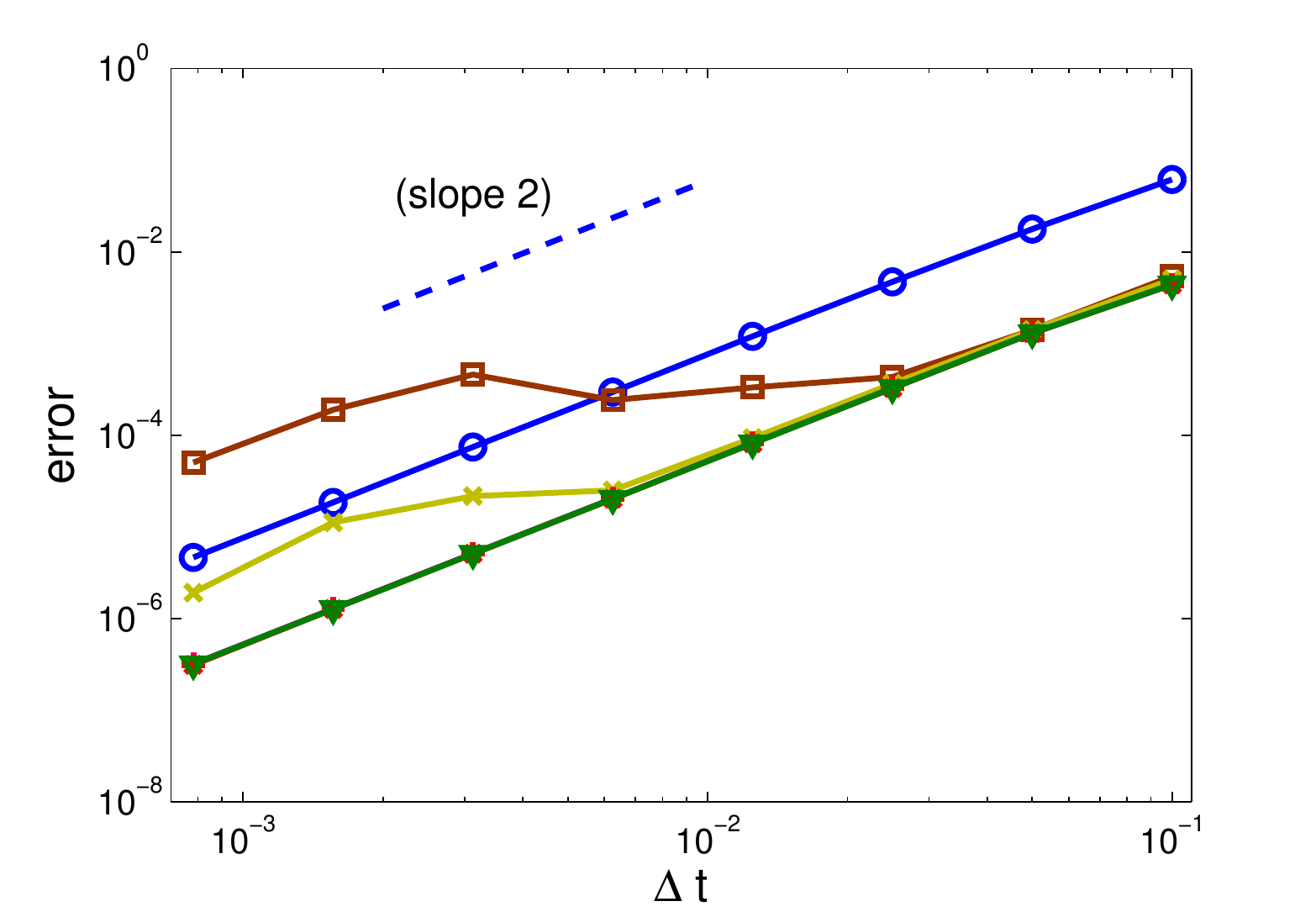,height=4.3cm,width=5.5cm}&\psfig{figure=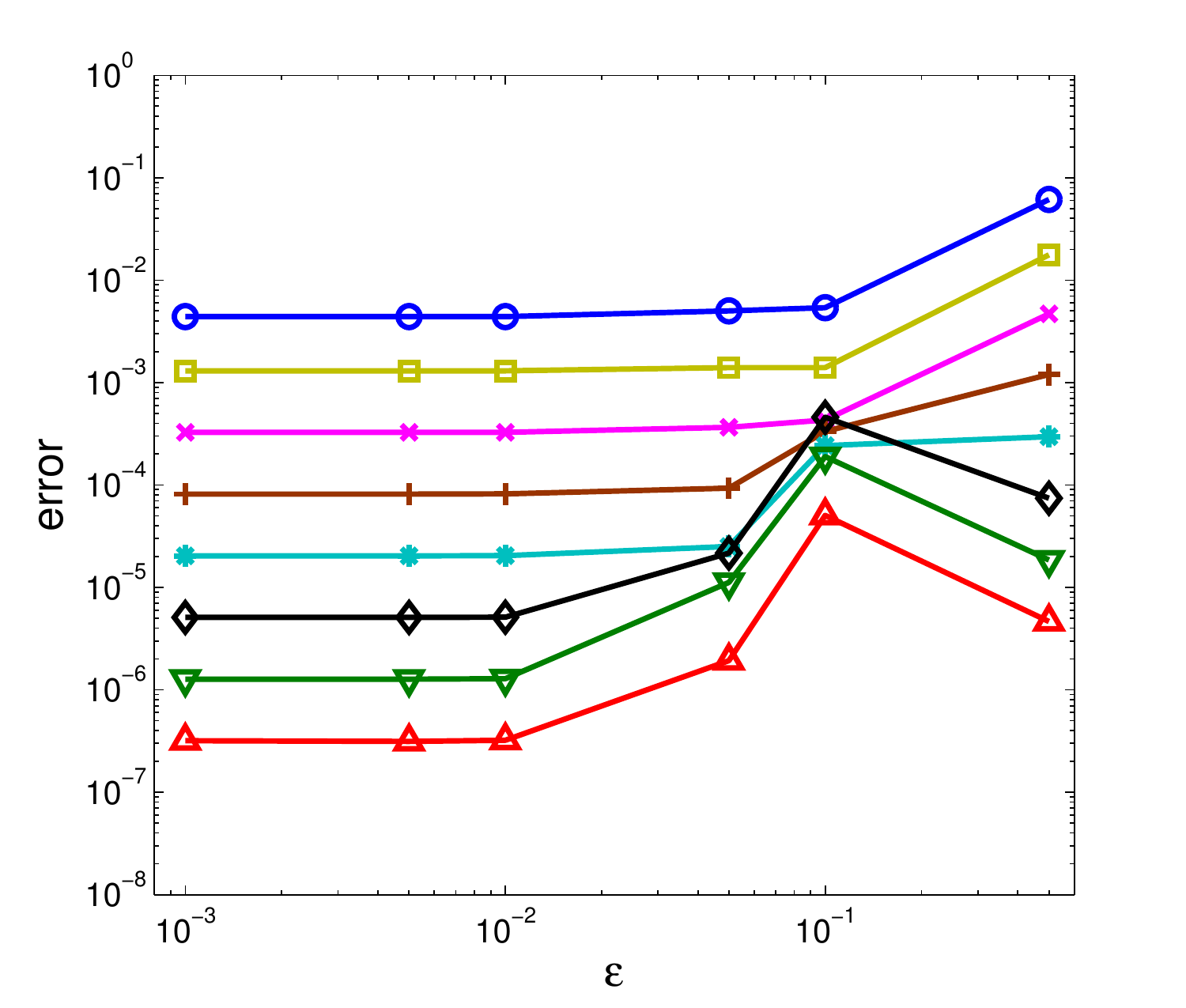,height=4.3cm,width=5.5cm}\\
\psfig{figure=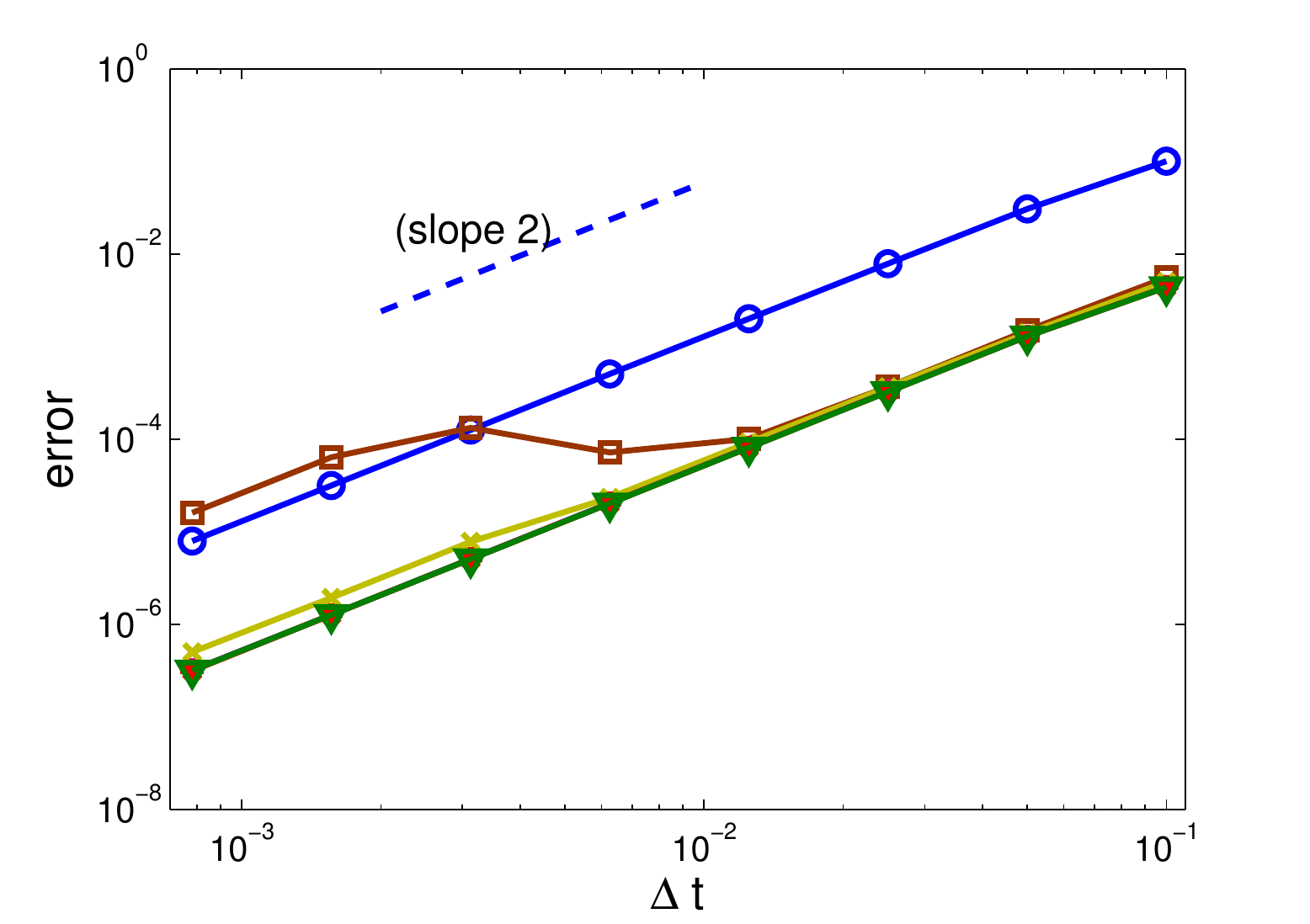,height=4.3cm,width=5.5cm}&\psfig{figure=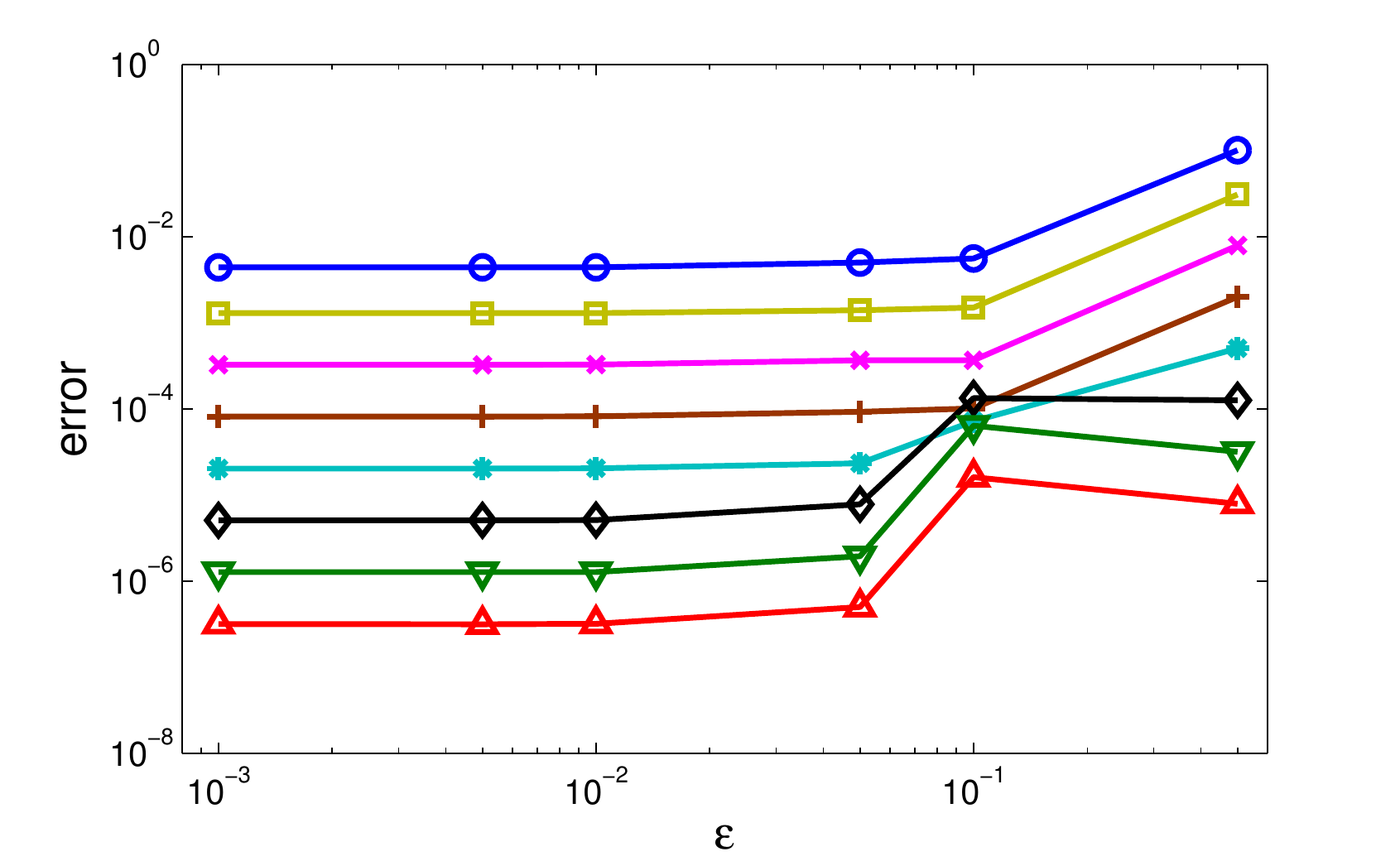,height=4.3cm,width=5.5cm}\\
\psfig{figure=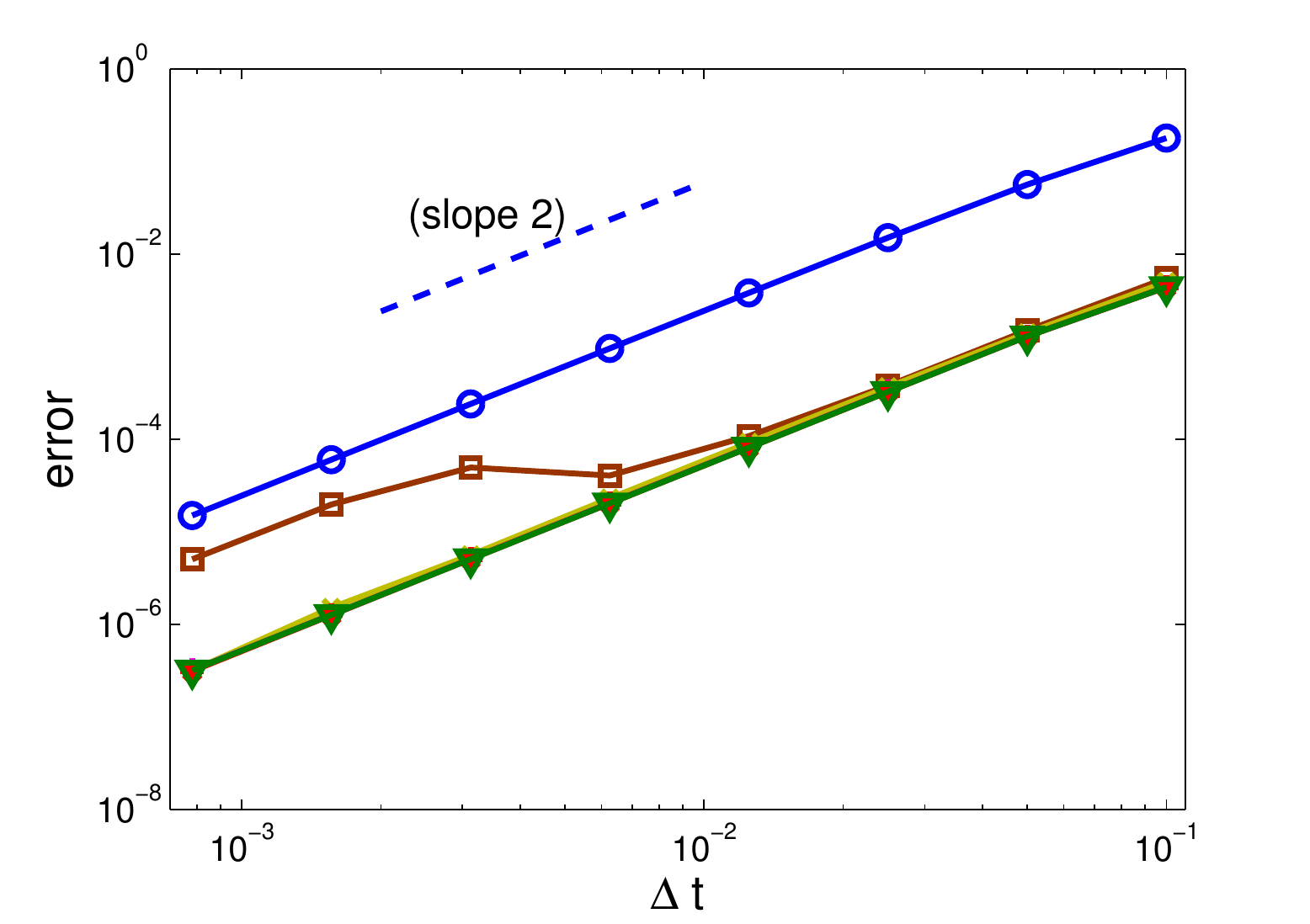,height=4.3cm,width=5.5cm}&\psfig{figure=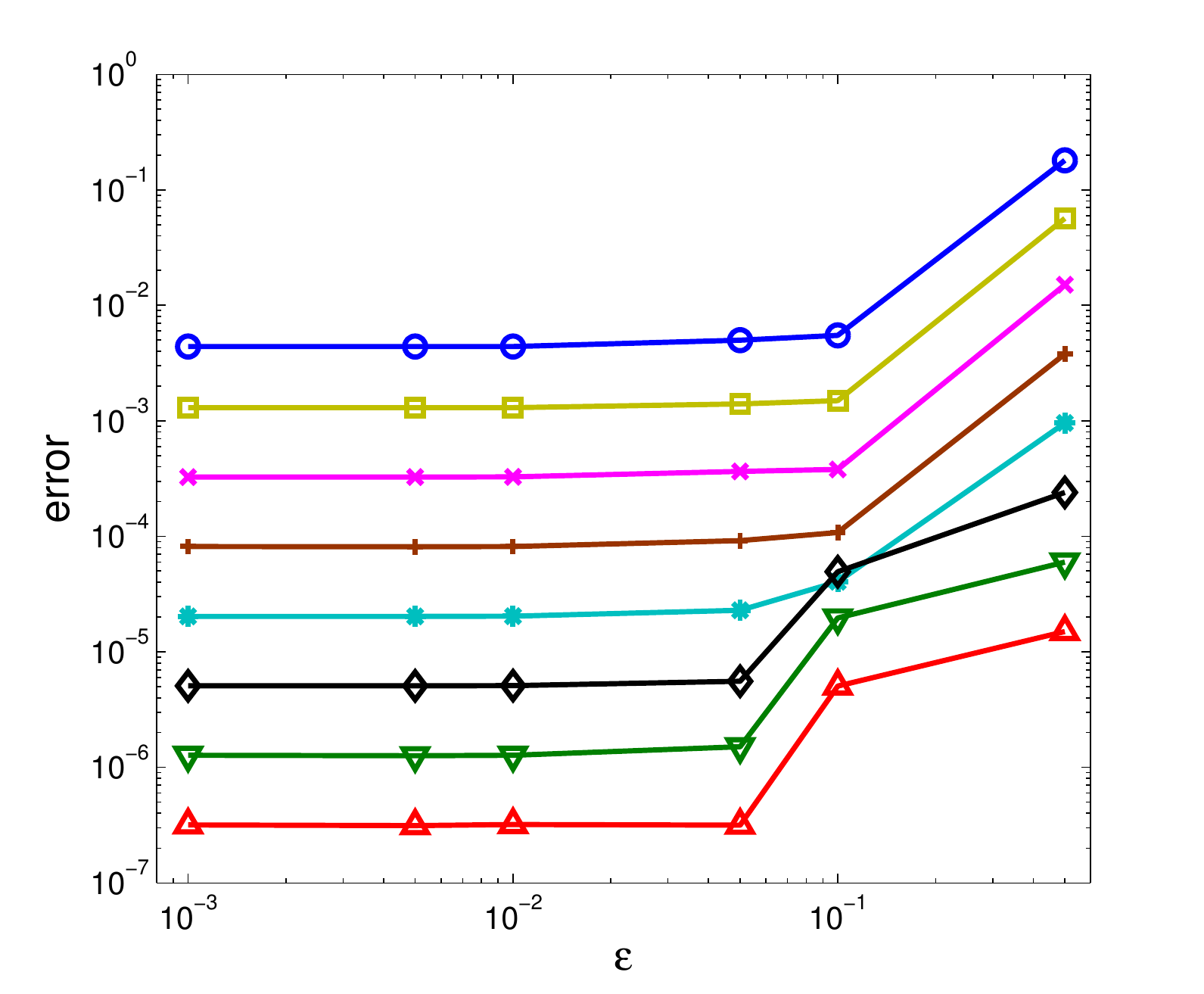,height=4.3cm,width=5.5cm}
\end{array}
$$
\caption{Temporal error of the UA2 method in Example I with respect to $\Delta t$ and $\eps$: results of using $U_2^\eps$ (first row); results of using $U_3^\eps$ (second row) results of using $U_4^\eps$ (third row);  results of using $U_5^\eps$ (last row).}\label{fig:err2}
\end{figure}


\emph{Example II: (linear with magnetic potential)}
We use the same setup in Example I but without nonlinearity and with a non-zero magnetic potential, i.e.
$$V_m=\frac{(x+1)^2}{1+x^2},\qquad \lambda=0.$$
 The temporal errors of the UA2 method at $t=0.5$ under different $\eps$ are shown in Fig. \ref{fig:errmlUA}.

\begin{figure}[h!]
$$
\begin{array}{cc}
\psfig{figure=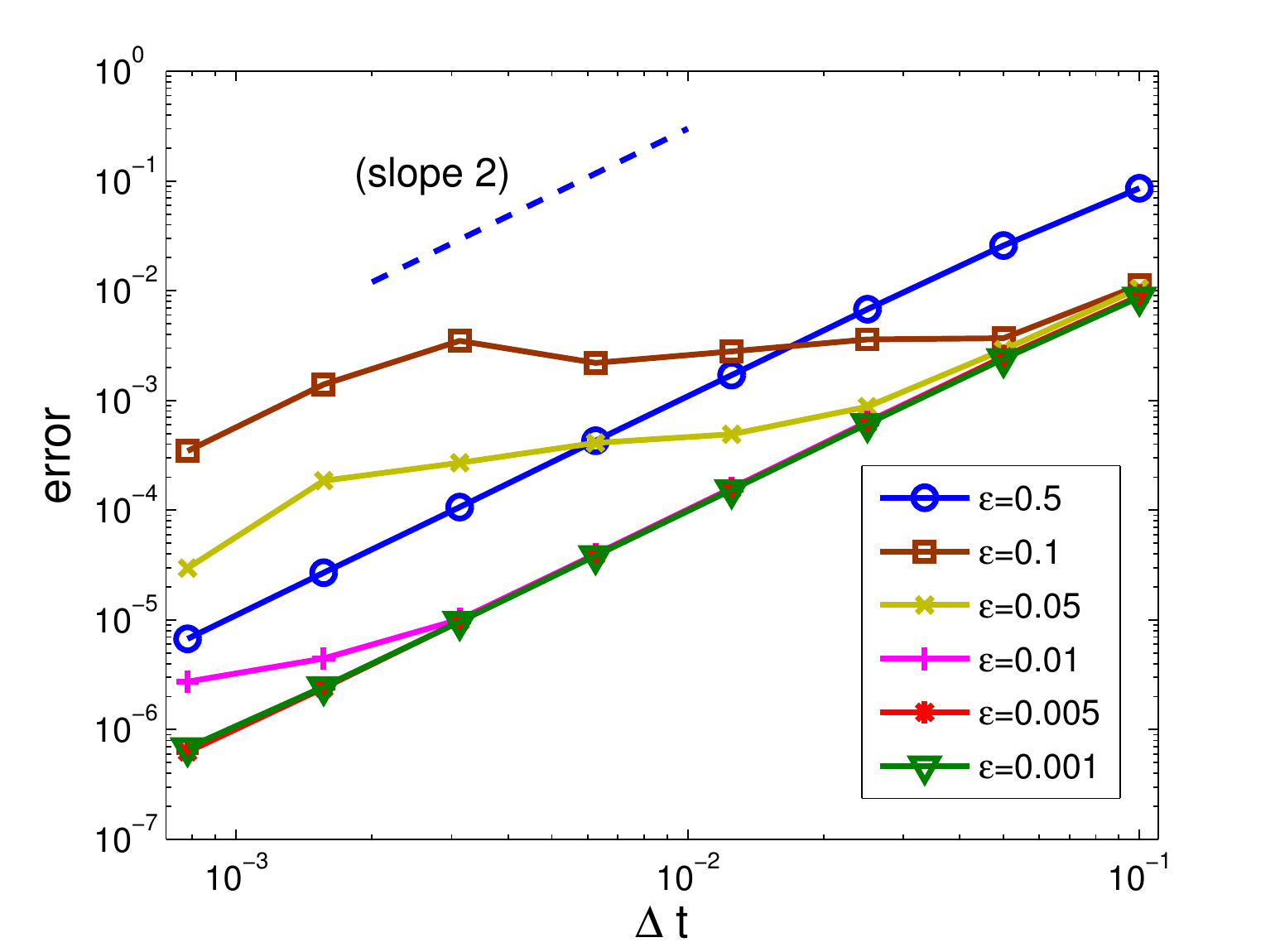,height=4.3cm,width=5.5cm}&\psfig{figure=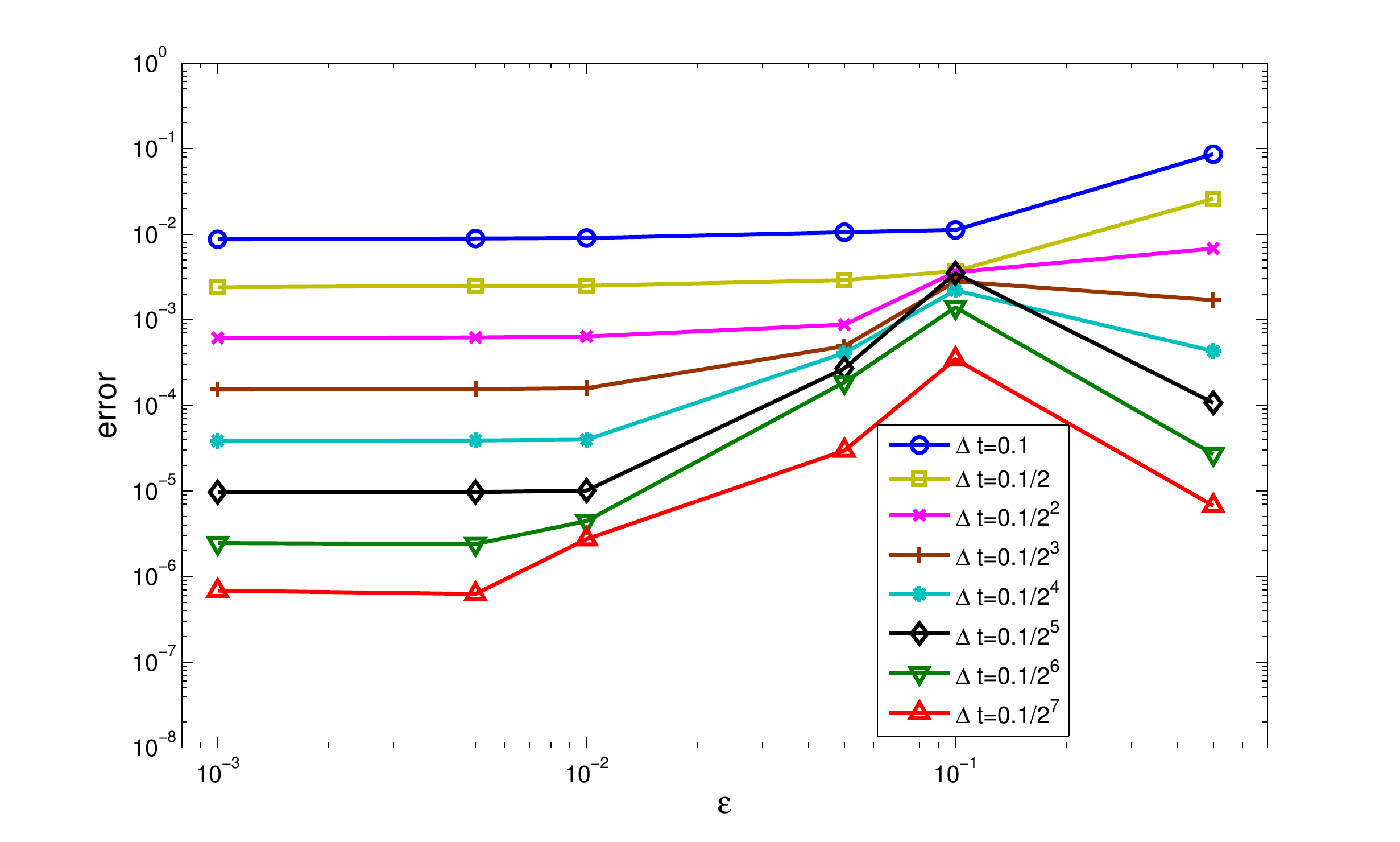,height=4.3cm,width=5.5cm}\\
\psfig{figure=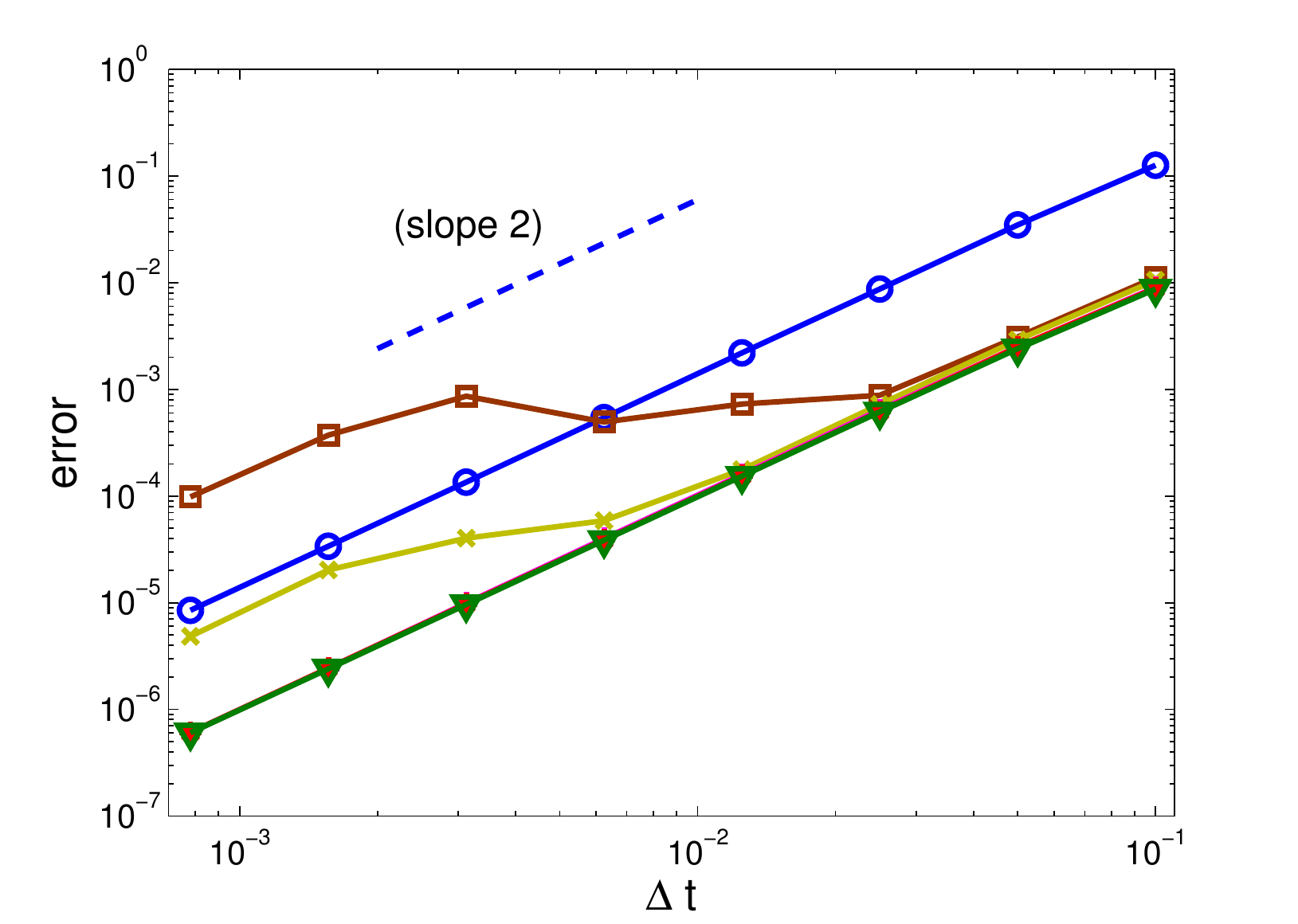,height=4.3cm,width=5.5cm}&\psfig{figure=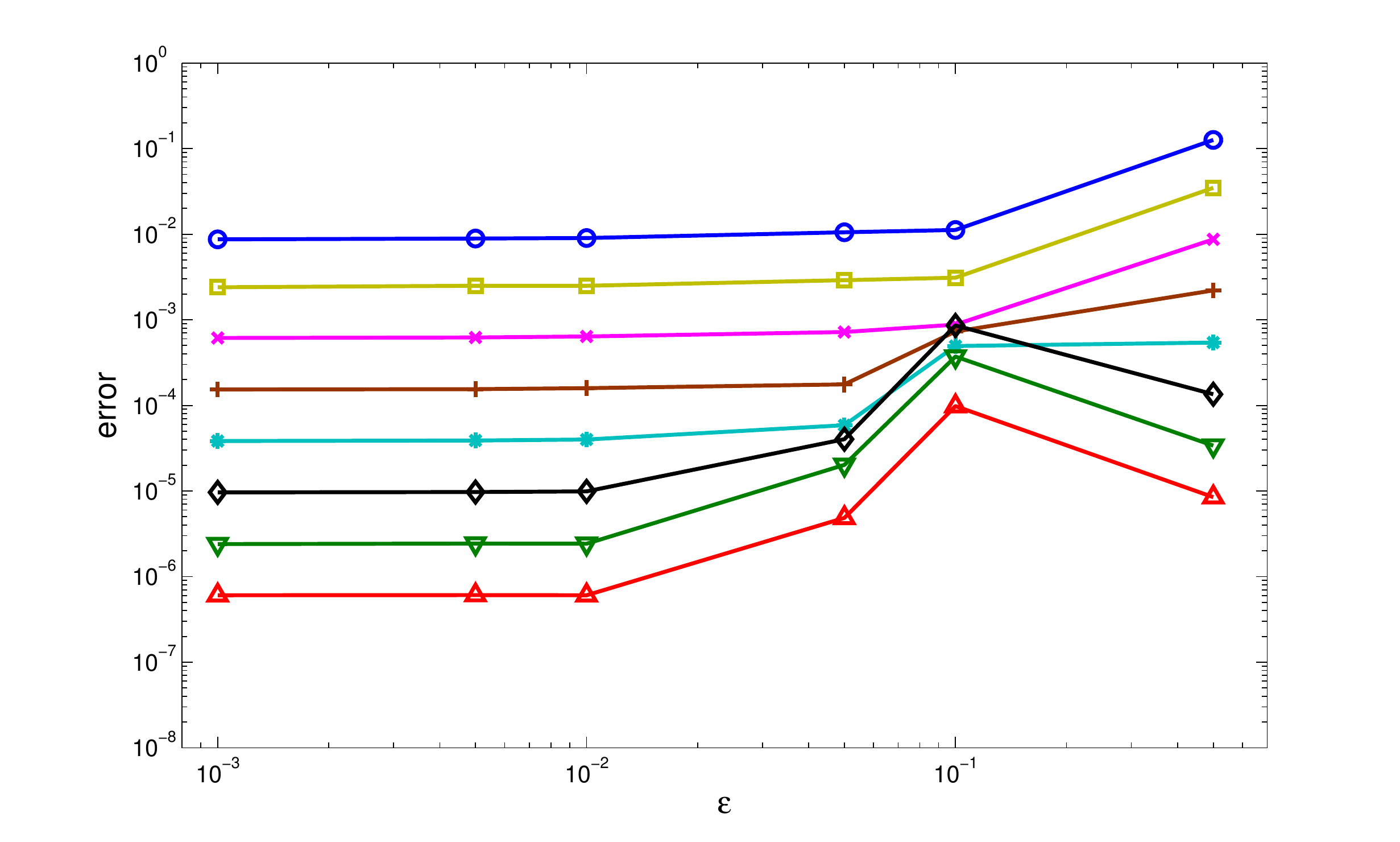,height=4.3cm,width=5.5cm}\\
\psfig{figure=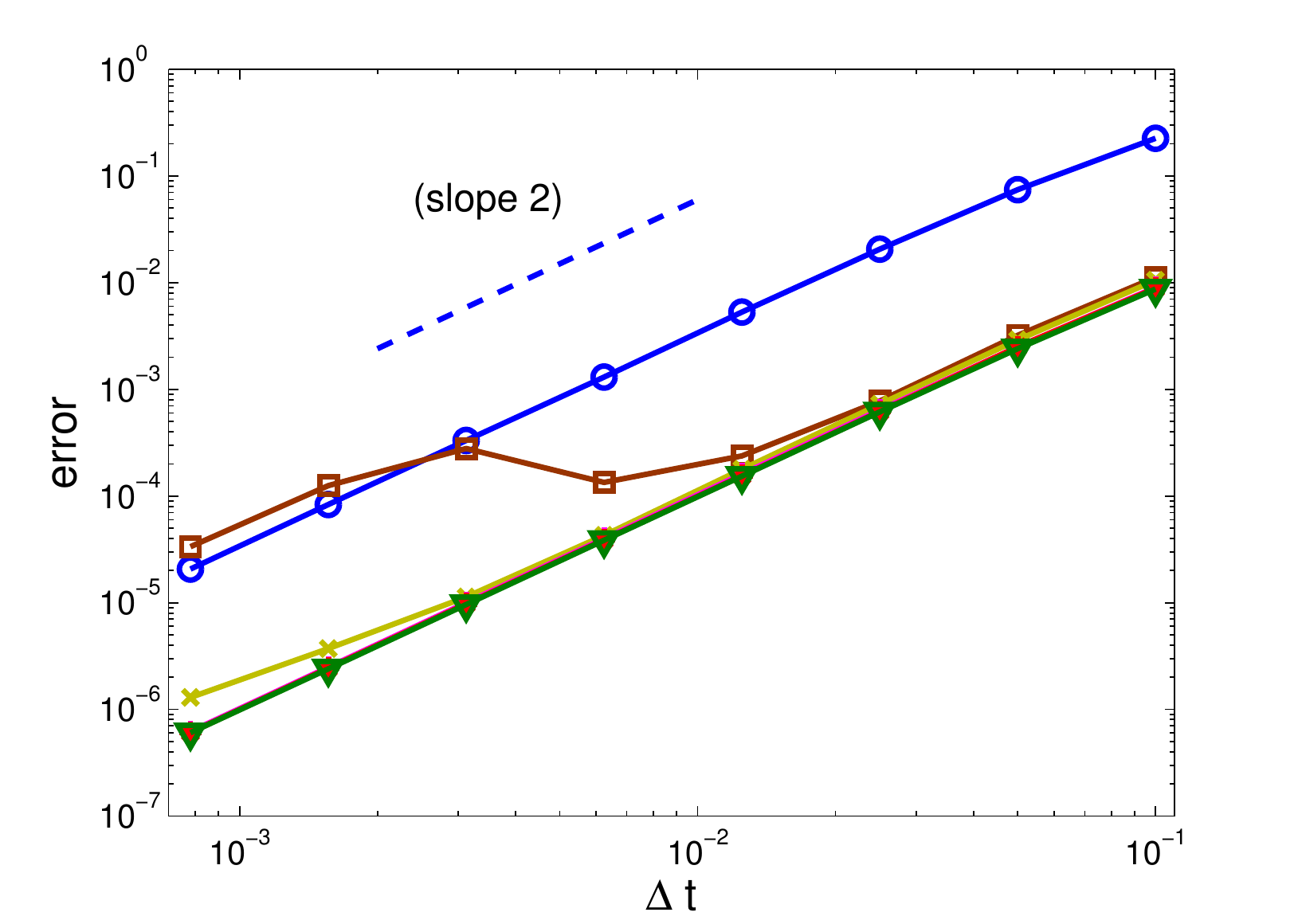,height=4.3cm,width=5.5cm}&\psfig{figure=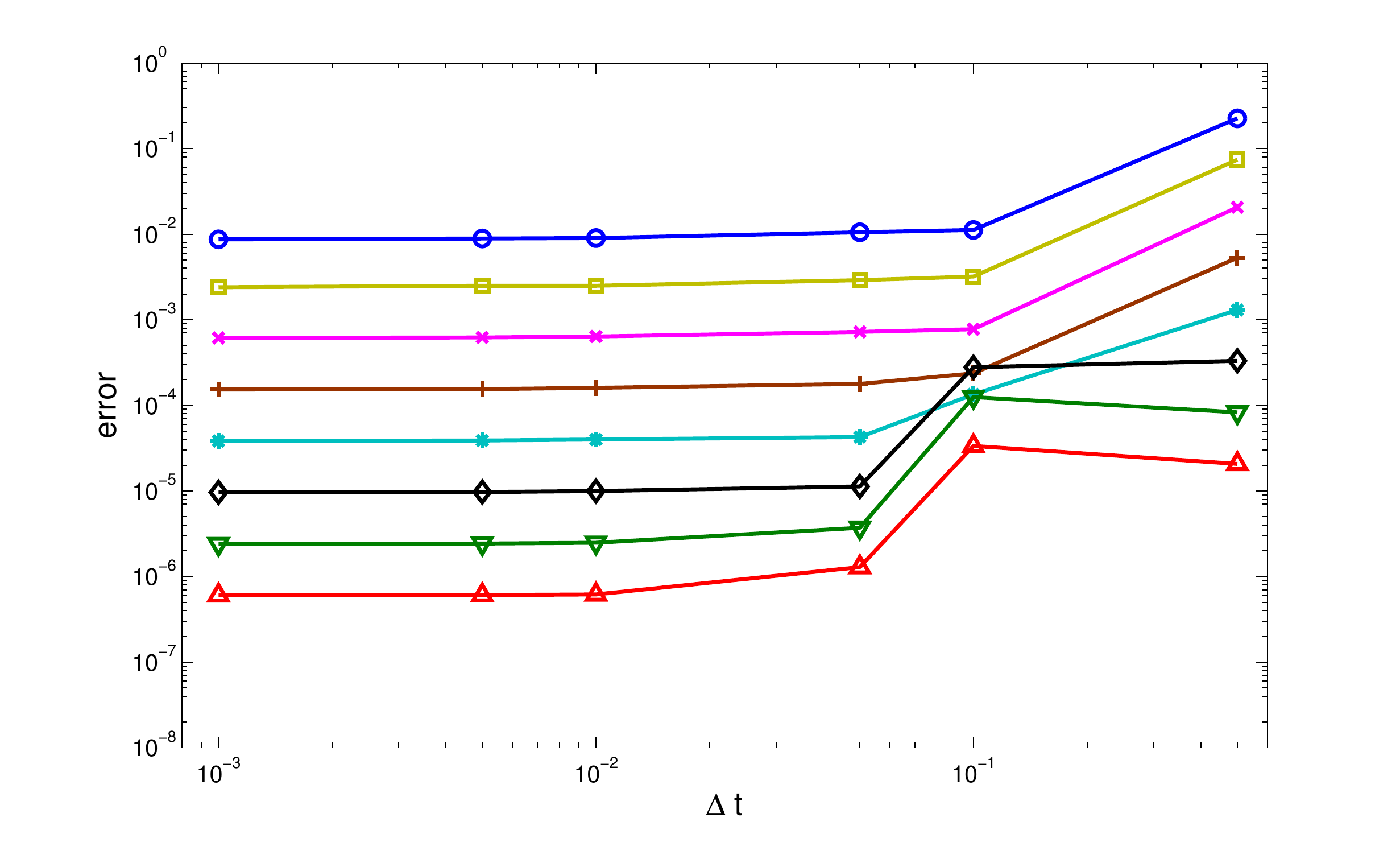,height=4.3cm,width=5.5cm}\\
\psfig{figure=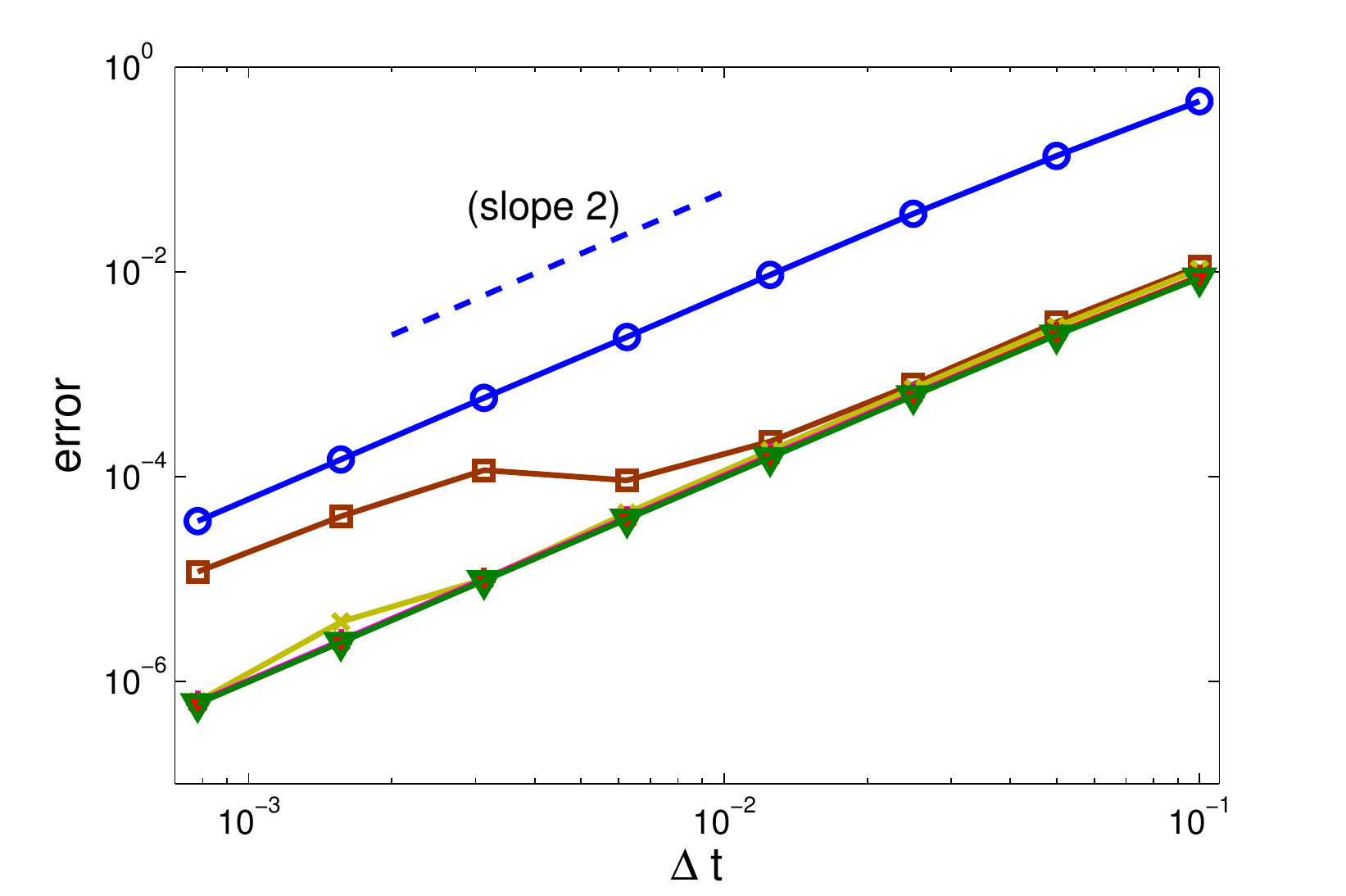,height=4.3cm,width=5.5cm}&\psfig{figure=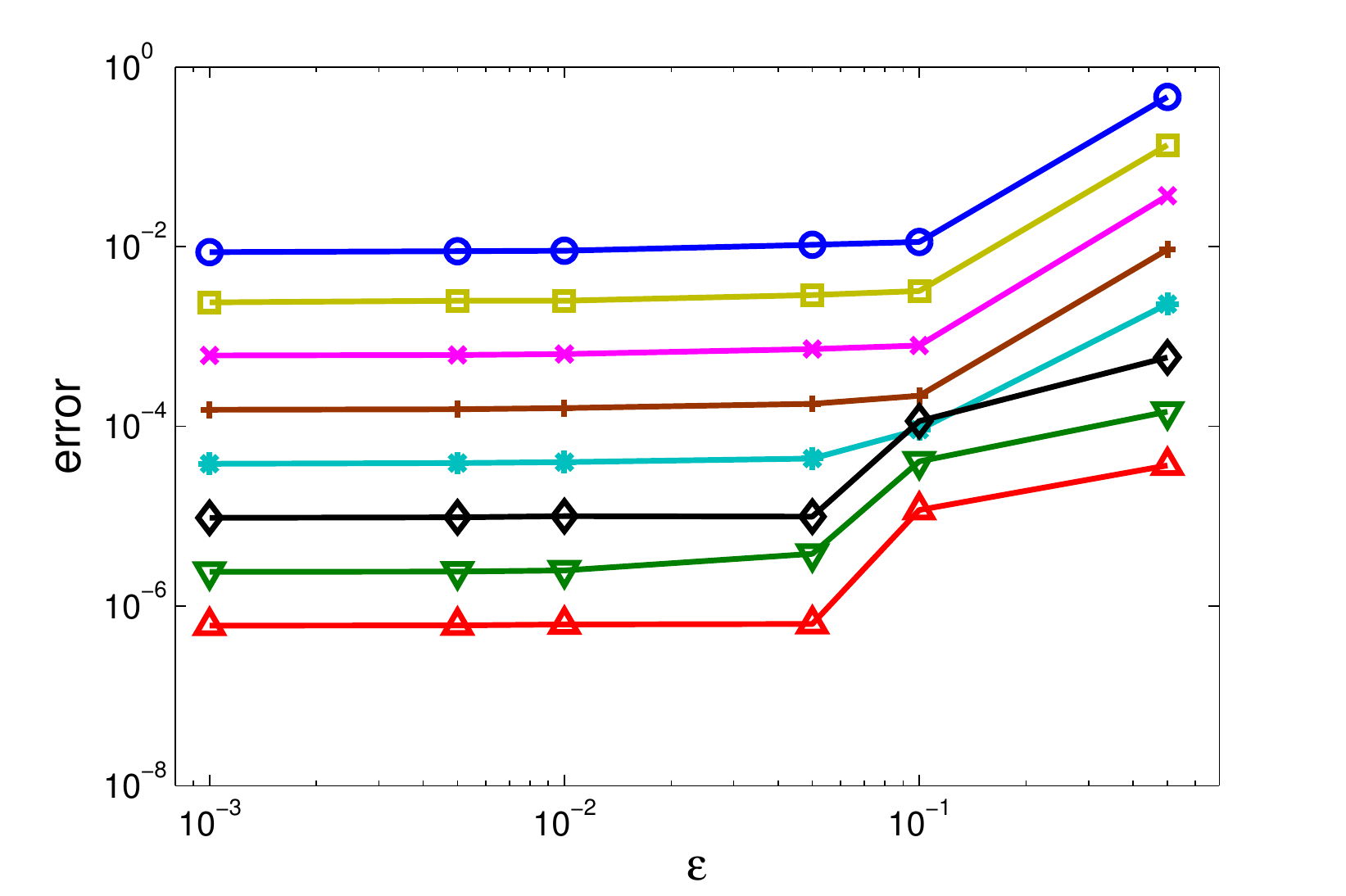,height=4.3cm,width=5.5cm}
\end{array}
$$
\caption{Temporal error of the UA2 method in Example II with respect to $\Delta t$ and $\eps$: results of using $U_2^\eps$ (first row); results of using $U_3^\eps$ (second row) results of using $U_4^\eps$ (third row);  results of using $U_5^\eps$ (last row).}\label{fig:errmlUA}
\end{figure}

\emph{Example III: (nonlinear with magnetic potential)}
We keep both the nonlinearity and magnetic potential, i.e.
$$V_m=\frac{(x+1)^2}{1+x^2},\qquad \lambda=0.5.$$
 The time discretization error of the UA2 method at $t=0.5$, with initial data $U^\eps_5$, is shown in Fig. \ref{fig:mnlUA}. The corresponding spatial errors of the UA2 method with respect to $\Delta x$ and $\Delta \tau$ are shown in Fig. \ref{fig:space}. The behaviour of the spatial errors of the UA1 method and of the other two numerical examples are similar, so the results are omitted here for brevity. 

 The error between the solution of the nonlinear Dirac equation (\ref{nlsw}) and the solution of the limit model (\ref{lm1}), i.e.
 $$\|\phi_1^\eps(t)-\phi_1(t)\|_{L^\infty}+\|\phi_2^\eps(t)-\phi_2(t)\|_{L^\infty},$$
 are shown in Fig. \ref{fig:rate} at $t=0.5$.

 \begin{figure}[h!]
$$
\begin{array}{cc}
\psfig{figure=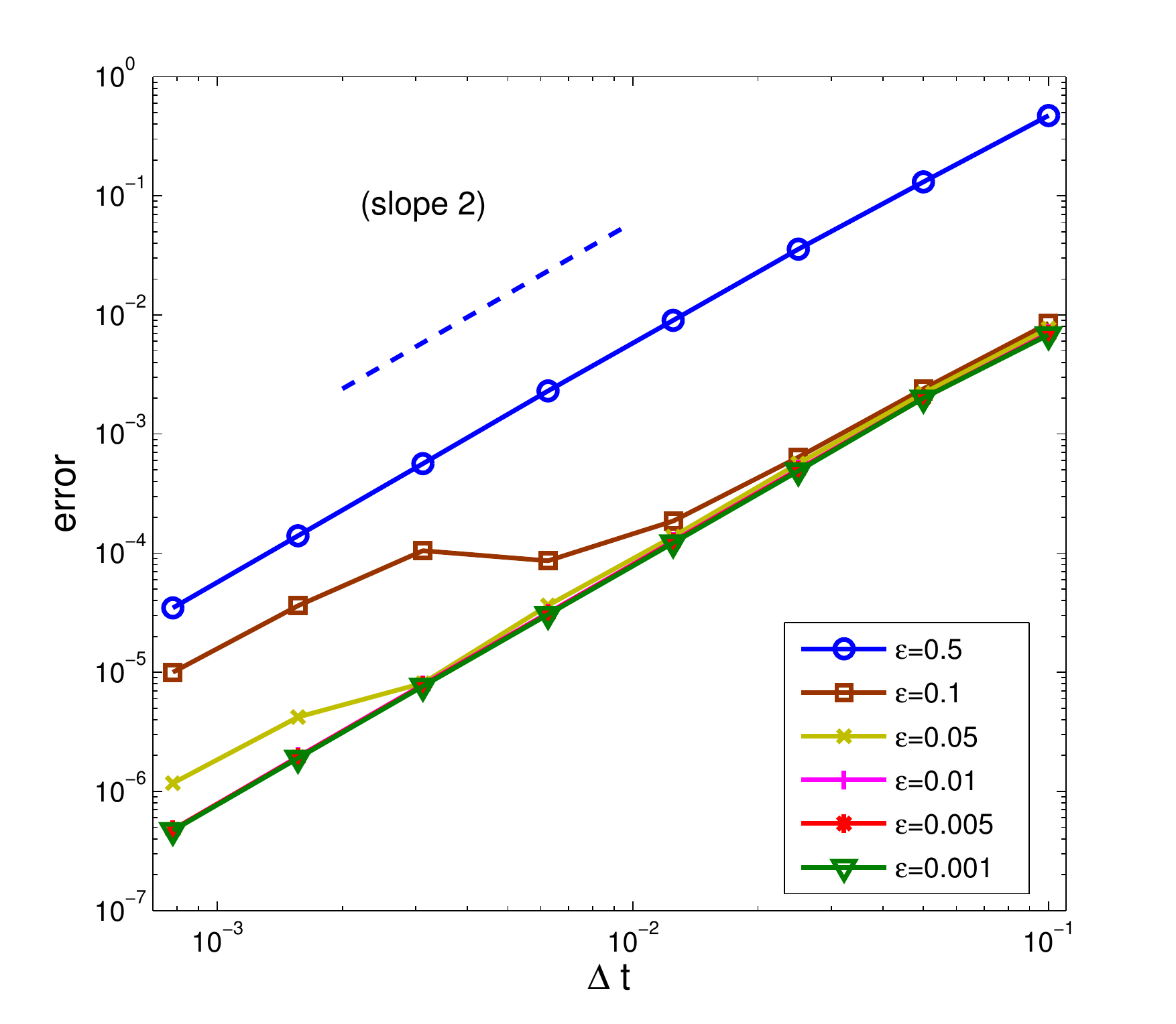,height=4.3cm,width=5.5cm}&\psfig{figure=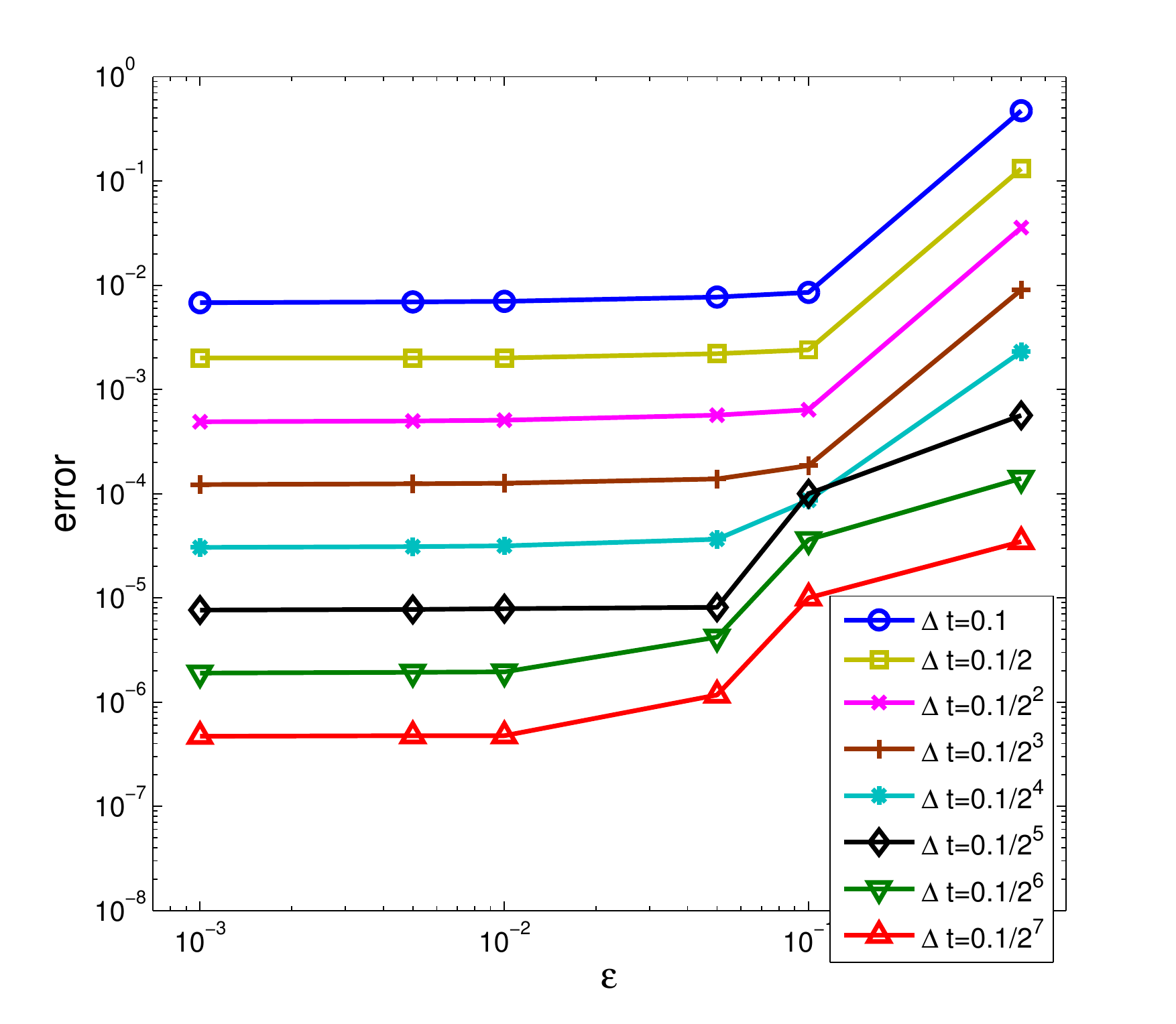,height=4.3cm,width=5.5cm}
\end{array}
$$
\caption{Temporal error of the UA2 method with $U_5^\eps$ in Example III with respect to $\Delta t$ and $\eps$: results of the nonlinear Dirac equation with magnetic potential case.}\label{fig:mnlUA}
\end{figure}

 \begin{figure}[h!]
$$
\begin{array}{cc}
\psfig{figure=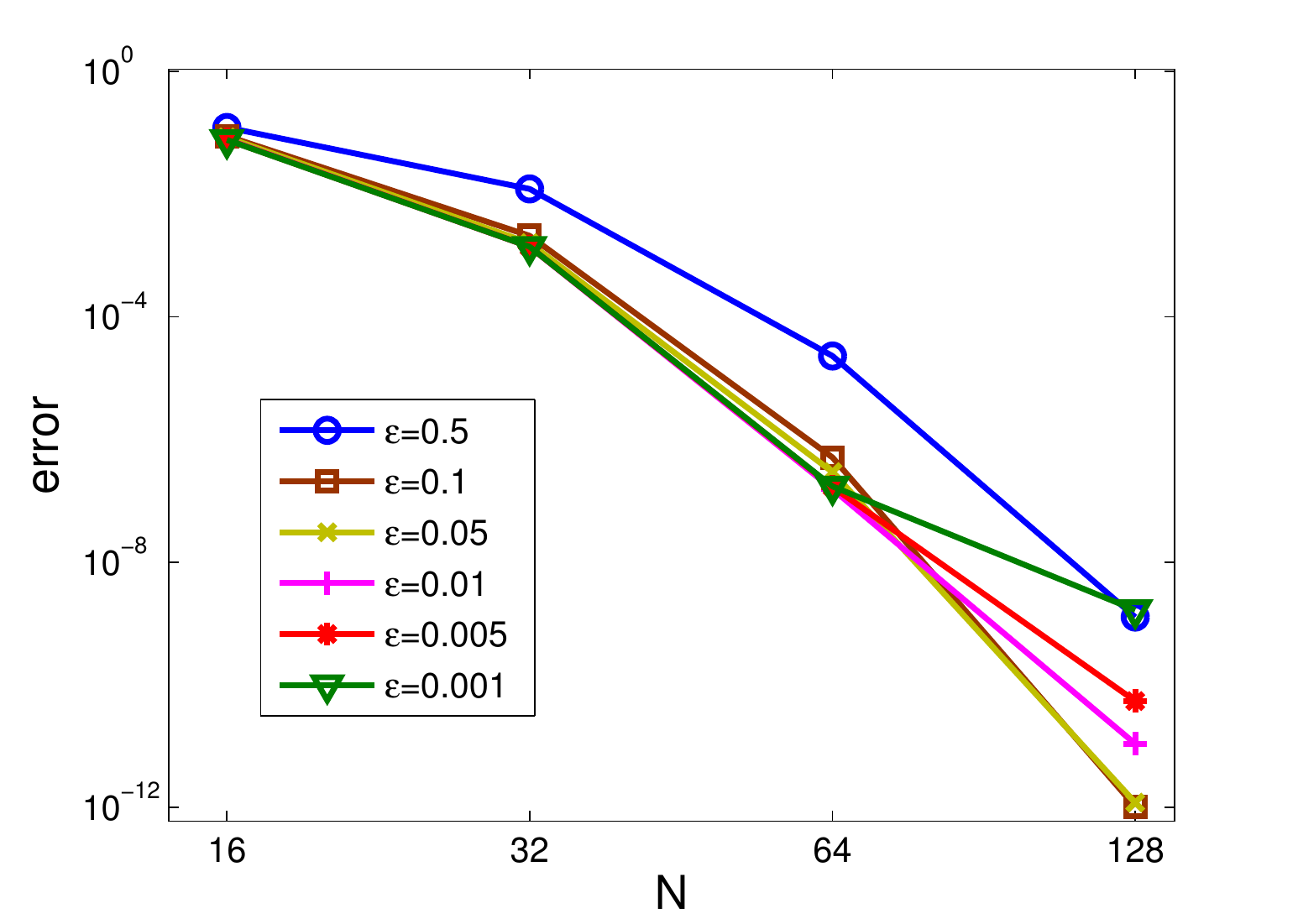,height=4.3cm,width=5.5cm}&\psfig{figure=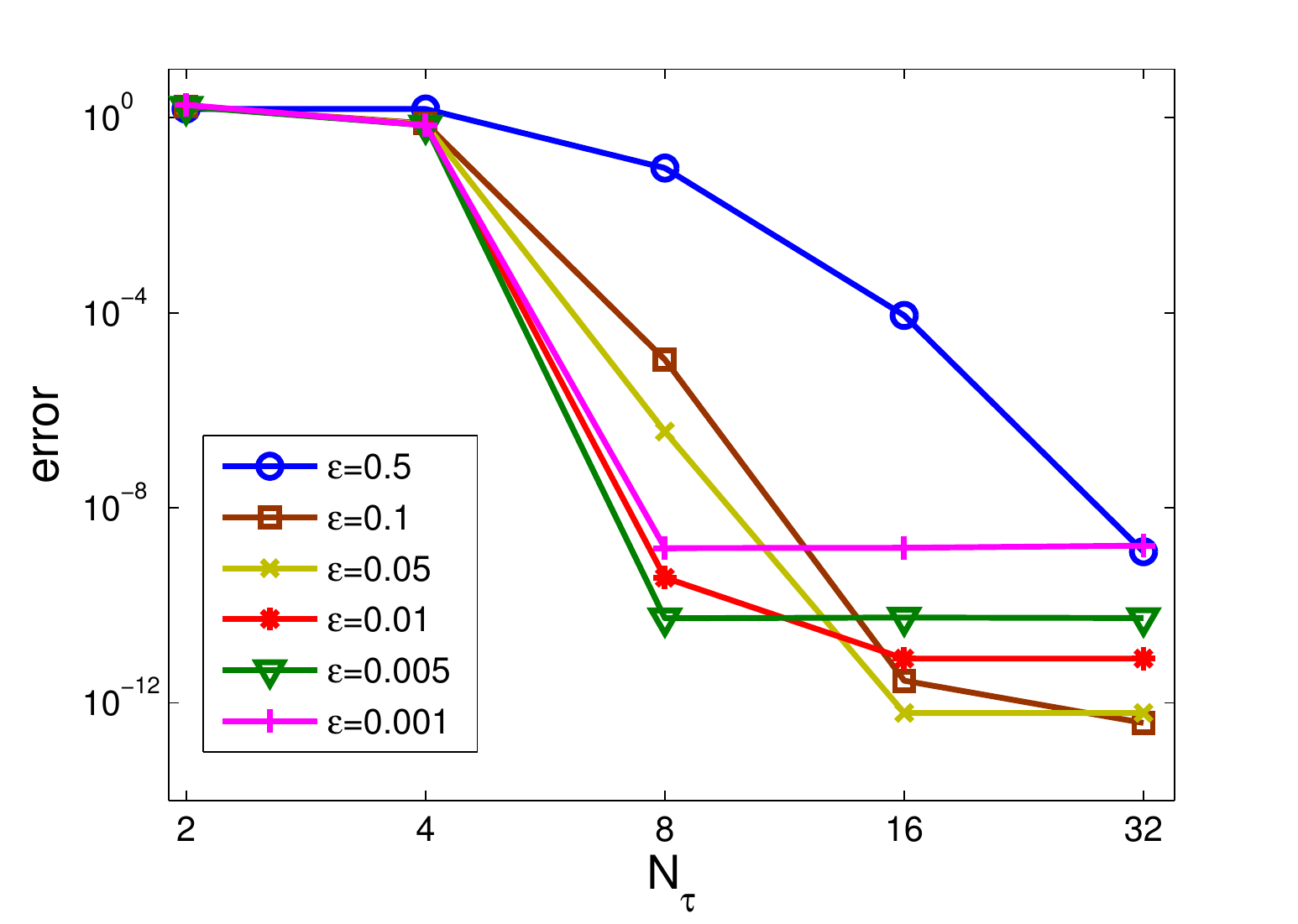,height=4.3cm,width=5.5cm}
\end{array}
$$
\caption{Spatial error of the UA2 method in Example III with respect to $N(=(b-a)/\Delta x)$ and $N_\tau$.}\label{fig:space}
\end{figure}
\begin{figure}[h!]
\centering
\psfig{figure=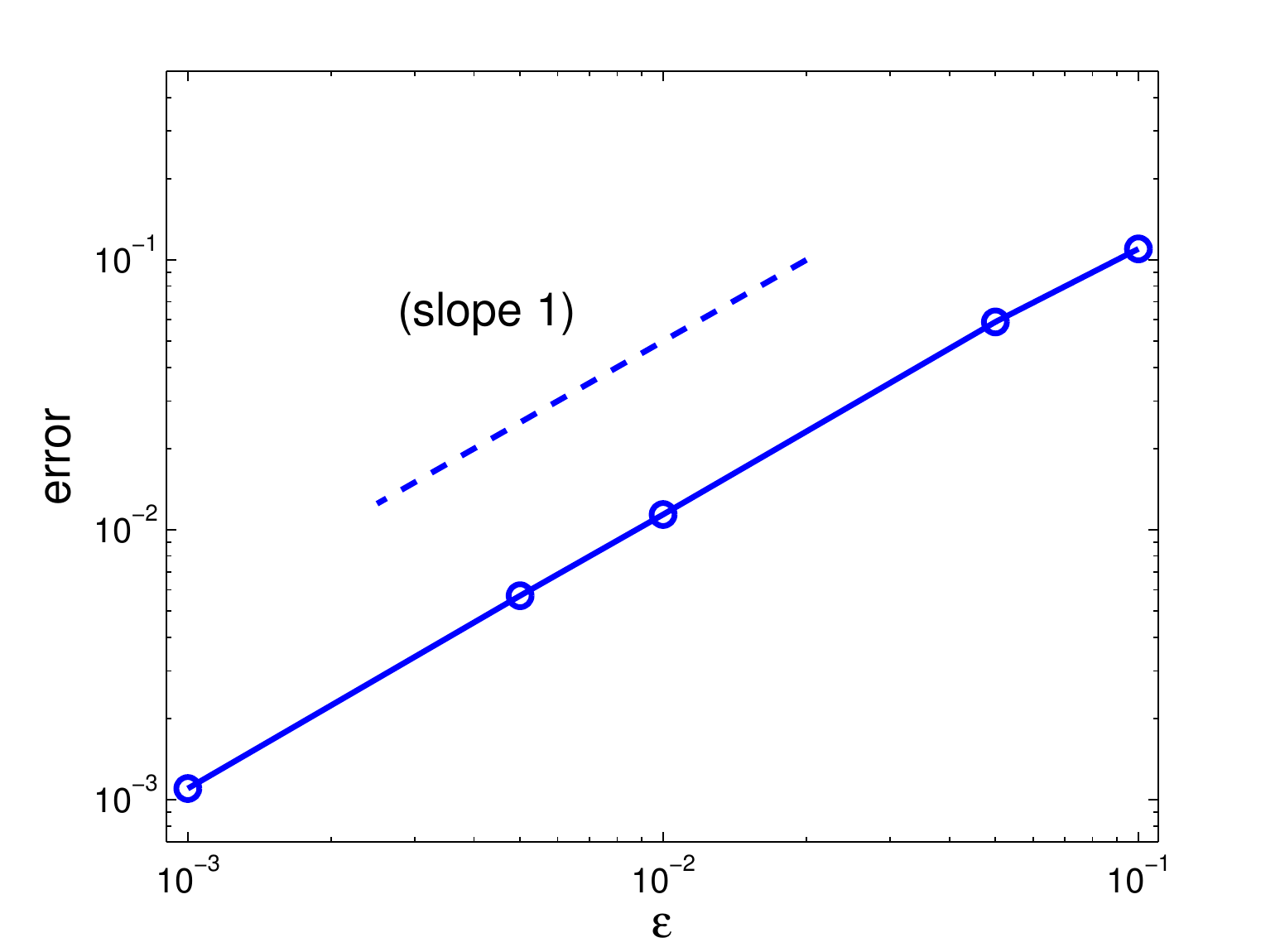,height=4.3cm,width=10cm}
\caption{The maximum error in solution between the nonlinear Dirac equation and the limit model (\ref{lm1}) in Example III.}\label{fig:rate}
\end{figure}

Based on the numerical results, we can have the following observation:

(i) The UA1 scheme with initial data $U_3^\eps$ shows uniformly first order temporal accuracy and the UA2 scheme with $U_5^\eps$ shows uniformly second order temporal accuracy in all the three cases. In space discretization, the UA schemes have uniformly spectral accuracy in both $x$ and the artificial $\tau$.


(ii) The convergence rate from the nonlinear Dirac equation to the limit model is of order $O(\eps)$, as $\eps\to0$.

\section{Conclusion}\label{sec:conc}
We proposed some uniformly accurately (UA) schemes for solving the nonlinear Dirac equation in the nonrelativistic limit regimes. Our approach is based on a suitable two-scale formulation which offers a general strategy for constructing UA schemes for a class of highly oscillatory problems involving two small scales. We derive correct initial data for this augmented formulation, using a Chapman-Enskog expansion. This allows us to construct a UA scheme with second order accuracy in time and spectral accuracy in space. Numerical tests were done to show the UA property. Our approach can also be applied to solve the oscillatory kinetic equations with diffusion scaling. This the subject of a work in progress.

\section*{Acknowledgements}
This work was supported by the French ANR project MOONRISE ANR-14-CE23-0007-01.
M. Lemou is supported by the Enabling Research EUROFusion project CfP-WP14-ER-01/IPP-03.
We would like to thank the editor and referees for their suggestions to improve the paper.
          %







          %


          \end{document}